\documentclass{amsart}
\usepackage{amssymb,mathrsfs,MnSymbol}
\usepackage{color}

\def\bC {\mathbf{C}}
\def\bD {\mathbf{D}}

\def\bN {\mathbf{N}}

\def\bR {\mathbf{R}}

\def\fH {\mathfrak{H}}
\def\fS {\mathfrak{S}}

\def\cA {\mathcal{A}}

\def\cD {\mathcal{D}}
\def\cE {\mathcal{E}}
\def\cF {\mathcal{F}}
\def\cG {\mathcal{G}}
\def\cH {\mathcal{H}}

\def\cL {\mathcal{L}}
\def\cM {\mathcal{M}}

\def\cP {\mathcal{P}}

\def\cS {\mathcal{S}}

\def\scrH{\mathscr{H}}
\def\scrK{\mathscr{K}}

\def\a {{\alpha}}

\def\de {{\delta}}
\def\eps {{\epsilon}}

\def\l {{\lambda}}
\def\L {{\Lambda}}
\def\si {{\sigma}}
\def\Si {{\Sigma}}

\def\om {{\omega}}

\def\d {{\partial}}
\def\grad {{\nabla}}
\def\Dlt {{\Delta}}

\def\rstr {{\big |}}
\def\indc {{\bf 1}}

\def\la {\langle}
\def\ra {\rangle}
\def \La {\bigg\langle}
\def \Ra {\bigg\rangle}

\newcommand{\Div}{\operatorname{div}}

\newcommand{\Dom}{\operatorname{Dom}}

\newcommand{\Supp}{\operatorname{supp}}

\newcommand{\Tr}{\operatorname{trace}}

\def\hb {{\hbar}}

\newcommand{\ba}{\begin{aligned}}
\newcommand{\ea}{\end{aligned}}

\newcommand{\be}{\begin{equation}}
\newcommand{\ee}{\end{equation}}

\newcommand{\lb}{\label}

\newtheorem{Thm}{Theorem}[section]

\newtheorem{Prop}[Thm]{Proposition}

\newtheorem{Lem}[Thm]{Lemma}

\begin{document}

\title[Mean-Field and Classical Limit for Coulomb Dynamics]{Mean-Field and Classical Limit for the $N$-Body Quantum Dynamics with Coulomb Interaction}

\author[F. Golse]{Fran\c cois Golse}
\address[F.G.]{CMLS, \'Ecole polytechnique, 91128 Palaiseau Cedex, France}
\email{francois.golse@polytechnique.edu}

\author[T. Paul]{Thierry Paul}
\address[T.P.]{CNRS \& CMLS, \'Ecole polytechnique, 91128 Palaiseau Cedex, France, \& Sorbonne Universit\'e, CNRS, Universit\'e de Paris, INRIA, Laboratoire Jacques-Louis Lions, 75005 Paris, France}
\email{thierry.paul@polytechnique.edu}

\begin{abstract}
This paper proves the validity of the joint mean-field and classical limit of the quantum $N$-body dynamics leading to the pressureless Euler-Poisson system for factorized initial data whose first marginal
has a monokinetic Wigner measure. The interaction potential is assumed to be the repulsive Coulomb potential. The validity of this derivation is limited to finite time intervals on which the Euler-Poisson 
system has a smooth solution that is rapidly decaying at infinity. One key ingredient in the proof is an inequality from [S. Serfaty, with an appendix of M. Duerinckx {\tt arXiv:1803.08345v3 [math.AP]}]. 
\end{abstract}




\maketitle


\section{Introduction}


The mean-field limit of the $N$-body Schr\"odinger or von Neumann equation for identical particles interacting via a repulsive Coulomb potential (such as electrons in a molecule) has been investigated in several
articles in the recent past \cite{ErdosYau,Pickl,RodSchlein}. The much simpler case of a bounded interaction potential has been treated earlier in Theorem 5.7 on p. 610 of \cite{Spohn} (see also \cite{BGM}). All 
these results are obtained in an asymptotic regime (large number $N$ of interacting particles), avoiding the semiclassical setting (i.e. assuming that the Planck constant $\hbar$ is of the same order of magnitude 
as the typical action per particle).

Classical dynamics, i.e. the differential system consisting of Newton's second law of motion written for each particle in a system, is known to describe the limit of Schr\"odinger or von Neumann equation in the 
semiclassical regime, i.e. assuming that the de Broglie wavelength of each particle is small compared to the typical length scale (or size) of the system: see for instance chapter VII in \cite{LL3} and section IV 
in \cite{LionsPaul}. The same is true of the quantum and classical mean-field dynamics respectively: see Theorems IV.2 and IV.4 in \cite{LionsPaul}. Thus, the uniformity as $\hbar\to 0$ of the mean-field limit for 
the quantum $N$-body dynamics would entail the validity of the mean-field limit for the $N$-body classical dynamics. At the time of this writing, the mean-field limit for the $N$-body classical dynamics has been 
proved rigorously for general initial data in the case of $C^{1,1}$ interaction potentials \cite{BraunHepp,Dobrushin}, or for a class of singular repulsive potentials \cite{HaurayJabin1,HaurayJabin2} whose singularity 
at the origin is weaker than that of the Coulomb potential. Other approaches \cite{Lazarovici,LazaroviciPickl} involve the idea of replacing either the Coulomb interaction, or the point charges with a conveniently 
chosen regularization thereof, where the regularization parameter is chosen so as to be vanishingly small, typically of order $O(N^{-\a})$ for some $\a>0$, in the large $N$ limit. 

Very recently, a new approach to the mean-field limit in classical dynamics with repulsive Coulomb interaction has led to a rigorous derivation of the pressureless Euler-Poisson system from Newton's second law
written for each element of a $N$-(point) particle system \cite{DS}. This new approach is based on a ``modulated energy method'' (reminiscent of \cite{Brenier} for instance, or of the relative entropy method in
\cite{Dafermos,Yau}, and of \cite{SerfatyJAMS} in the Ginzburg-Landau setting, with the case of Riesz potentials being treated in \cite{Duerinckx}), assuming that the target solution is smooth on some finite time 
interval. A very remarkable feature of the approach in \cite{DS} is that it allows handling directly the Coulomb potential (and even other singular potentials, viz. Riesz potentials) without any unphysical modification 
(such as truncation of either the potential or the charge empirical density). Unfortunately this approach seems at present limited to the case of monokinetic solutions of the Vlasov-Poisson equation --- in other words,
to deriving the Euler-Poisson rather than the Vlasov-Poisson system. Another potential drawback of this approach if one has in mind to extend it to a quantum setting is that it is couched in terms of phase-space 
empirical measures, in other words of Klimontovich solutions of the Vlasov equation, as is the case of the pioneering work of \cite{NeunzertWick} and of the proofs in \cite{BraunHepp,Dobrushin}. 

We briefly recall 
the notion of Klimontovich solution of the Vlasov equation, which is a core idea in all these contributions. If $q_1(t),\ldots,q_N(t)$ and $p_1(t),\ldots,p_N(t)$ are respectively the positions and momenta of the $N$ 
particles, the system of Newton's second law of motion written for 
each particle (assumed to be of mass $1$), i.e.
$$
\dot q_j(t)=p_j(t)\,,\qquad \dot p_j(t)=-\tfrac1N\sum_{k=1}^N\grad V(q_j(t)-q_k(t))\,,\qquad j=1,\ldots,N\,,
$$
where $V/N$ is the (pairwise) interaction potential, is equivalent, in the case where $V\in C^{1,1}(\bR^d)$, to the fact that the phase-space empirical measure of the $N$-particle system
$$
\frac1N\sum_{j=1}^N\de_{q_j(t),p_j(t)}(x,\xi)
$$
is a measure (weak) solution of the Vlasov equation
$$
\d_tf+\xi\cdot\grad_xf-\grad_xV[f](t,x)\cdot\grad_\xi f=0\,,\quad\text{ with }V[f](t,x):=\iint_{\bR^{2d}}V(x-y)f(t,dyd\eta)\,.
$$
Solutions of the Vlasov equation of this type are referred to as ``Klimontovich solutions''. (The $1/N$ scaling factor multiplying the interaction potential is typical of the mean-field limit in classical mechanics.) Then 
the validity of the mean-field limit in this setting is equivalent to the weakly continuous dependence of the measure solution to the Vlasov equation in terms of its initial data.

Various mathematical tools have been constructed recently in order to prove the uniformity as $\hbar\to 0$ of the quantum mean-field convergence rate. These tools involve quantum analogues of the Wasserstein
or Monge-Kantorovich distances in optimal transport \cite{FGMouPaul,FGTPaulCRAS,FGPaulPulvi}, or a quantum analogue of the notion of phase-space empirical measure, or of Klimontovich solutions to the 
Vlasov equation \cite{FGTPaulEmpir}. All these tools have produced satisfying results in the case of interaction potentials that are at least $C^{1,1}$. Putting them to use in the case of singular potentials (even 
less singular than the Coulomb potential) remains to be done. The interest for this uniformity issue comes from earlier works \cite{GraffiMartiPulvi,PezzoPulvi}, where partial results have been obtained with more
traditional analytical tools.

In the present paper, we explain how one key inequality in \cite{DS} can be used to prove the uniformity as $\hbar\to 0$ of the mean-field limit for the \textit{quantum} dynamics of point particles interacting via a 
repulsive Coulomb potential. While Proposition 2.3 in \cite{DS} can be used essentially as a black box in the present paper, the remaining part of the proof differs noticeably from \cite{DS}, in the first place 
because of the formalism of Klimontovich solutions systematically used in \cite{DS}: see Remark (1) following Serfaty's inequality \eqref{SSIneq} for a more detailed discussion of this point. One could follow the 
analysis in \cite{DS}, replacing classical Klimontovich solutions with their quantum analogue constructed in \cite{FGTPaulEmpir}, at the cost of rather heavy mathematical technique. Instead, we have chosen 
to take advantage of one simplifying feature in \cite{DS}, which allows controlling the convergence rates of the mean-field (large $N$) and semiclassical (small $\hbar$) regimes in terms of the first two equations 
in the BBGKY hierarchy. This simplifying feature is discussed in detail in Remarks (2) and (3) following Serfaty's inequality \eqref{SSIneq}, and the connection with the BBGKY hierarchy is explained in the 
following Remark (4).  


\section{Presentation of the Problem. Main Results}


Denote by $V(z):=1/4\pi|z|$ the repulsive Coulomb potential in $\bR^3$. For each  integer $N\ge 2$, consider the $N$-body quantum Hamiltonian
$$
\scrH_N:=\sum_{j=1}^N-\tfrac12\hbar^2\Dlt_{x_j}+\frac1N\sum_{1\le j<k\le N}V(x_j-x_k)
$$
with the weak coupling constant $1/N$ that is characteristic of the mean-field limit for $N$ bosons. With the notation $\fH:=L^2(\bR^3)$ and $\fH_N:=L^2((\bR^3)^N)\simeq\fH^{\otimes N}$, the unbounded operator
$\scrH_N$ with domain $\Dom(\scrH_N):=H^2((\bR^3)^N)$ is self-adjoint on $\fH_N$ (see Lemma 5 and Theorem 1 in \cite{KatoAMS}), and therefore generates a unitary group on $\fH_N$ by Stone's theorem
(see section VIII.4 in \cite{RS1}).

The quantum $N$-particle dynamics in the case of pure quantum states is defined in terms of the Cauchy problem for the $N$-body Schr\"odinger equation, with unknown the $N$-particle wave function 
$\Psi_N\equiv\Psi_N(t,x_1,\ldots,x_N)\in\bC$:
\be\lb{NSchro}
i\hbar\d_t\Psi_N=\scrH_N\Psi_N\,,\qquad\Psi_N\rstr_{t=0}=\Psi_N^{in}\,.
\ee
The generalized solution of the Cauchy problem is
$$
\Psi_N(t,\cdot)=e^{-it\scrH_N/\hbar}\Psi_N^{in}\,,\qquad t\in\bR\,,
$$
which is defined for all $\Psi_N^{in}\in\fH_N$. The function $t\mapsto\Psi_N(t,\cdot)$ so obtained belongs to $C(\bR_+;\fH_N)$.

Alternately, in the case of mixed quantum states, one considers instead the  Cauchy problem for the $N$-body von Neumann equation, with unknown $N$-particle density operator $R_N\equiv R_N(t)\in\cD(\fH_N)$.
The algebra of bounded operators on the Hilbert space $\cH$ is denoted $\cL(\cH)$. The two-sided ideal of trace class (resp. Hilbert-Schmidt) operators is denoted $\cL^1(\cH)$ (resp. $\cL^2(\cH)$). The operator,
trace and Hilbert-Schmidt norms are denoted respectively $\|\cdot\|$, $\|\cdot\|_1$ and $\|\cdot\|_2$. The sets of density operators on $\fH_N$ and of symmetric density operators on $\fH_N$ are denoted respectively
$$
\ba
\cD(\fH_N):=&\{R\in\cL(\fH_N)\text{ s.t. }R=R^*\ge 0\,,\quad\Tr_{\fH_N}(R)=1\}\,,
\\
\cD_s(\fH_N):=&\{R\in\cD(\fH_N)\text{ s.t. }U_\si RU_\si^*=R\text{ for all }\si\in\fS_N\}\,,
\ea
$$
where
$$
U_\si\Psi_N(x_1,\ldots,x_N):=\Psi_N(x_{\si^{-1}(1)},\ldots,x_{\si^{-1}(N)})\,.
$$
The $N$-body von Neumann equation takes the form
\be\lb{NvN}
i\hbar\d_tR_N(t)=[\scrH_N,R_N(t)]\,,\qquad R_N\rstr_{t=0}=R_N^{in}\,,
\ee
and its generalized solution is 
$$
R_N(t)=e^{-it\scrH_N/\hbar}R_N^{in}e^{+it\scrH_N/\hbar}\,,\qquad t\in\bR\,.
$$
One easily checks that
$$
R_N^{in}\in\cD_s(\fH_N)\implies R_N(t)\in\cD_s(\fH_N)\quad\text{ for all }t\in\bR\,.
$$
The connection between \eqref{NSchro} and \eqref{NvN} is explained by the following observation: if the initial data in \eqref{NvN} is $R_N^{in}=|\Psi_N^{in}\ra\la\Psi_N^{in}|$, then the solution $t\mapsto R_N(t)$ of 
\eqref{NvN} is $R_N(t)=|\Psi_N(t)\ra\la\Psi_N(t)|$ for all $t\in\bR$, where $t\mapsto\Psi_N(t)$ is the solution of \eqref{NSchro}. Throughout this paper, we use systematically the Dirac bra-ket notation\footnote{If 
$\psi\in L^2(\bR^d)$, the notation $|\psi\ra$ designates $\psi$ viewed as a vector in $L^2(\bR^d)$, while $\la\psi|$ designates the continuous linear functional on $L^2(\bR^d)$ defined by the formula
$$
\la\psi|\phi\ra:=\int_{\bR^d}\overline{\psi(x)}\phi(x)dx\,.
$$}.

\bigskip
The Euler-Poisson system with unknown $(\rho,u)\equiv(\rho(t,x),u(t,x))\in[0,+\infty)\times\bR^3$ takes the form
\be\lb{EulP}
\left\{
\ba
{}&\d_t\rho+\Div_x(\rho u)=0\,,&&\qquad\rho\rstr_{t=0}=\rho^{in}\,,
\\
&\d_tu+u\cdot\grad_xu+\grad_xV\star_x\rho=0\,,&&\qquad u\rstr_{t=0}=u^{in}\,.
\ea
\right.
\ee
The following result is classical, and stated in \cite{DS}.

\begin{Prop}\lb{P-ExLocEulP}
Assume that $\rho^{in}\in H^{2m}(\bR^3)$ satisfies
\be\lb{NormaRhoin}
\rho^{in}(x)\ge 0\text{ for a.e. }x\in\bR^3\,,\quad\text{ and }\quad\int_{\bR^3}\rho^{in}(y)dy=1\,,
\ee
and let $u^{in}\in L^\infty(\bR^3)$ be such that $\grad_xu^{in}\in H^{2m}(\bR^3)$. 

Then, there exists $T\equiv T[\|\rho^{in}\|_{H^{2m}(\bR^3)}+\|\grad_xu^{in}\|_{H^{2m}(\bR^3)}]>0$, and a solution $(\rho,u)$ of the pressureless Euler-Poisson system \eqref{EulP} such that
$$
t\mapsto\rho(t,\cdot)\text{ and }t\mapsto\d_{x_j}u_k(t,\cdot)\text{ belong to }C([0,T],H^{2m}(\bR^3))
$$
for each $j,k=1,2,3$, while $u\in L^\infty([0,T]\times\bR^3)$. Besides, for all $t\in[0,T]$, one has
\be\lb{NormaRhot}
\rho(t,x)\ge 0\text{ for a.e. }x\in\bR^3\,,\quad\text{ and }\quad\int_{\bR^3}\rho(t,y)dy=1\,.
\ee
In particular, if $m=2$, then $\rho\in L^\infty([0,T]\times\bR^3)$, while $\d^\a_xu_k\in L^\infty([0,T]\times\bR^3)$ for each $k=1,2,3$ and each $\a\in\bN^3$ such that $|\a|\le 3$.
\end{Prop}

\bigskip
Before we state our main result, we recall a few elementary facts on trace-class operators. 

If $R\in\cL^1(L^2(\bR^d))$, it has an integral kernel (which we henceforth abusively denote $R\equiv R(x,y)$) satisfying the following property: the function $z\mapsto R(x,x+z)$ belongs to $C_b(\bR^d;L^1(\bR^d))$.
In particular, the restriction of $R$ to the diagonal, i.e. the function $\rho(x):=R(x,x)$ is a well-defined element of $L^1(\bR^d)$, called the \textit{density function} of $R$. See for instance Example 1.18 in chapter X, 
\S 1 of \cite{Kato}, together with Lemma 2.1 (1) in \cite{BGM} or footnote 1 on pp. 61--62 of \cite{FGTPaulARMA}.

Henceforth, we denote $D_x:=-i\grad_x$. Assume that $R\in\cL^1(L^2(\bR^d))$ satisfies
$$
R=R^*\ge 0\quad\text{ and }\quad\Tr(R^{1/2}|D_x|^2R^{1/2})<\infty\,.
$$
For each test function $\phi\in C_b(\bR^d)$ and each $k=1,\ldots,d$, one defines a signed Radon measure $J^k$ on $\bR^d$ by the formula
\be\lb{DefCurrent}
\int_{\bR^d}\phi(x)J^k(dx):=\Tr((\phi\vee\tfrac12\hbar D_k)R)\,.
\ee
(Here we have used the notation
\be\lb{AntiCom}
A\vee B=AB+BA
\ee
to designate the anti-commutator of the operators $A$ and $B$.) The (signed) measure-valued vector field $(J^1,\ldots,J^d)$ so defined is henceforth referred to as the \textit{current} of $R$.

If $T$ is a trace-class operator on $L^2(\bR^{d_1}\times\bR^{d_2})\simeq L^2(\bR^{d_1})\otimes L^2(\bR^{d_2})$, one can define its partial trace $\Tr_1(T)\in\cL^1(L^2(\bR^{d_1}))$ by the identity
$$
\Tr_{L^2(\bR^{d_1})}(A\Tr_1(T))=\Tr_{L^2(\bR^{d_1}\times\bR^{d_2})}((A\otimes I_{L^2(\bR^{d_2})})T)\,.
$$
If $T\in\cD_s(\fH_N)$, one defines the marginals of $T$ as follows: for each $k=1,\ldots,N-1$, the $k$-th marginal of $T$, denoted by $T_{:k}$, is defined as $\Tr_1(T)\in\cD(\fH_k)$, viewing $T$ as a trace-class 
operator on $\fH_k\otimes\fH_{N-k}\simeq\fH_N$. Equivalently, $T_{:k}$ is defined by duality as follows:
$$
\Tr_{\fH_k}(T_{:k}A)=\Tr_{\fH_N}\left(T\left(A\otimes I_{\fH_{N-k}}\right)\right)\,,\quad\text{ for all }A\in\cL(\fH_k)\,.
$$

The Wigner transform \cite{LionsPaul} of $R$ is defined as
$$
W_{\hbar}[R](x,\xi):=\tfrac1{(2\pi)^d}\int_{\bR^d}R(x+\tfrac12\hbar y,x-\tfrac12\hbar y)e^{-i\xi\cdot y}dy\,.
$$
Since $\cL^1(L^2(\bR^d))\subset\cL^2(L^2(\bR^d))$, the Wigner transform $W_{\hbar}[R]\in L^2(\bR^{2d})$ by the Plancherel theorem.

\subsection{From the $N$-body quantum dynamics to the Euler-Poisson system}

Our main result is the following statement.

\begin{Thm}\lb{T-MFSCLim}
Let $\rho^{in}\in H^4(\bR^3)$ satisfy \eqref{NormaRhoin}, and let $u^{in}\in L^\infty(\bR^3)$ such that $\grad u^{in}\in H^4(\bR^3)$. Let $T>0$ and let $(\rho,u)$ be a solution  of the pressureless Euler-Poisson 
system \eqref{EulP} such that $\rho$ and $\d_x^\a u_j\in L^\infty([0,T]\times\bR^3)$ for all $k=1,2,3$ and all $\a\in\bN^3$ such that $|\a|\le 3$.

Let $R_{\hbar,N}^{in}=(R_{\hbar}^{in})^{\otimes N}$, where $R_{\hbar}^{in}\in\cD(\fH)$ satisfies 
\be\lb{TightR0}
\sup_{\hbar\in(0,1)}\Tr_\fH((R_{\hbar}^{in})^{1/2}|\hbar D_x|^4(R_{\hbar}^{in})^{1/2})<\infty\,,
\ee
and
\be\lb{Monokin0}
\Tr_\fH((R_{\hbar}^{in})^{1/2}|\hbar D_x-u^{in}(x)|^2(R_{\hbar}^{in})^{1/2})\to 0\text { as }\hbar\to 0\,,
\ee
while the density function $\rho_{\hbar}^{in}$ of $R_{\hbar}^{in}$ satisfies 
\be\lb{EPrho0}
\iint_{\bR^3\times\bR^3}\frac{(\rho_{\hbar}^{in}(x)-\rho^{in}(x))(\rho_{\hbar}^{in}(y)-\rho^{in}(y))}{|x-y|}dxdy\to 0\quad\text{ as }\hbar\to 0\,.
\ee
Let $R_{\hbar,N}(t):=e^{-it\scrH_N/\hbar}R_{\hbar,N}^{in}e^{+it\scrH_N/\hbar}$ be the generalized solution of \eqref{NvN} with initial data $R_{\hbar,N}^{in}$, and let $R_{\hbar,N:1}(t)$ be its first marginal density.
Denote by $\rho_{\hbar,N:1}(t,\cdot)$ and $J_{\hbar,N:1}(t,\cdot)$ respectively the density function and the current of $R_{\hbar,N:1}(t)$.

In the limit as $\frac1N+\hbar\to 0$ and for each $t\in[0,T]$, one has\footnote{We adopt here the terminology from \cite{Malliavin}. The linear space of bounded Radon measures on $\bR^d$ is the dual of 
the linear space $C_0(\bR^d)$ of continuous functions on $\bR^d$ vanishing at infinity, equipped with the $L^\infty$ norm (Theorem II.6.6 in \cite{Malliavin}). The weak topology of bounded Radon measures 
is the topology defined by the family of seminorms $\mu\mapsto|\la\mu,\phi\ra|$ where $\phi$ runs through $C_0(\bR^d)$. The narrow topology of bounded Radon measures is the topology defined by the 
family of seminorms $\mu\mapsto|\la\mu,\psi\ra|$ where $\psi$ runs through $C_b(\bR^d)$, the space of bounded continuous functions on $\bR^d$.}:
\be\lb{Cvrhot}
\rho_{\hbar,N:1}(t,\cdot)\to\rho(t,\cdot)\text{ narrowly in }\cP(\bR^3)\,,
\ee
while
\be\lb{CvJt}
J_{\hbar,N:1}(t,\cdot)\to\rho u(t,\cdot)\text{ for the narrow topology of signed Radon measures on }\bR^3\,.
\ee
Besides
\be\lb{Monokint}
\Tr_\fH(R_{\hbar,N:1}(t)^{1/2}|\hbar D_x-u(t,x)|^2R_{\hbar,N:1}(t)^{1/2})\to 0
\ee
and
\be\lb{Wignert}
W_{\hbar}[R_{\hbar,N:1}(t)]\to\rho(t,x)\de(\xi-u(t,x))\quad\text{ in }\cS'(\bR^3)
\ee
for each $t\in[0,T]$ as $\frac1N+\hbar\to 0$.
\end{Thm}

\smallskip
Notice that $\hbar$ and $1/N$ are \textit{completely independent} vanishingly small parameters in the Theorem above. In other words, Theorem \ref{T-MFSCLim} does \textit{not} refer to a distinguished 
asymptotic regime (such as $\hb=N^{-1/3}$ as in the case of the mean-field limit for fermions: see for instance the introduction of \cite{PRSS}). Theorem \ref{T-MFSCLim} is the analogue of the main result
in \cite{FGTPaulARMA} (Theorem 2.6 of \cite{FGTPaulARMA}) for nearly monokinetic quantum states --- i.e. states satisfying the condition \eqref{Monokint} --- and in the case where the interaction is a 
repulsive Coulomb potential, at variance with Theorem 2.6 of \cite{FGTPaulARMA} which considers only interaction potentials in $C^{1,1}(\bR^3)$. However, the method used in \cite{FGTPaulARMA} avoids 
the restriction to nearly monokinetic initial data, i.e. assumption \eqref{Monokin0}.

\smallskip
Another remark is that the choice of $R_{\hbar,N}^{in}=(R_{\hbar}^{in})^{\otimes N}$ is consistent with the Bose-Einstein statistics whenever $R_{\hbar}^{in}$ is a pure state, i.e. is of the form
$$
R_{\hbar}^{in}=|\psi_{\hbar}^{in}\ra\la\psi_{\hbar}^{in}|\,,\quad\text{ so that }R_{\hbar,N}^{in}=|\Psi_{\hbar,N}^{in}\ra\la\Psi_{\hbar,N}^{in}|\,,\quad\text{ with }\Psi_{\hbar,N}^{in}= (\psi_{\hbar}^{in})^{\otimes N}\,.
$$
Also, as mentioned above, the $1/N$ scaling of the interaction potential $V$ in the Hamiltonian $\scrH_N$ is typical of the mean-field limit for a system of $N$ bosons. Indeed, this scaling provides a balance 
between the kinetic and potential energies for the choice $R^{in}_{\hbar,N}=(R^{in}_{\hbar})^{\otimes N}$.

\subsection{Consequences of Theorem \ref{T-MFSCLim}}

The most important consequence of Theorem \ref{T-MFSCLim} is the uniformity as $\hbar\to 0$ of the mean-field limit of the quantum $N$-body dynamics to the Hartree equation in the large $N$ limit, 
in the case of a repulsive Coulomb interaction and in the ``near monokinetic'' regime.

We recall that the Hartree equation (in the formalism of density operators) is
\be\lb{Hartree}
i\hbar\d_tR_{\hbar}(t)=[-\tfrac12\hbar^2\Dlt_x+V\star\rho_{\hbar}(t,\cdot),R_{\hbar}(t)]\,,\qquad R_{\hbar}(0)=R_{\hbar}^{in}\,,
\ee
where $\rho_{\hbar}(t,\cdot)$ is the density function of $R_{\hbar}(t)$. 

The existence and uniqueness of the solution to this Cauchy problem has been studied in detail by \cite{ChadGlass,Bove}. We recall\footnote{See Proposition 5.5 and section 6 in \cite{Bove}, discarding the
exchange potential $B_{EX}$ and setting $B(T)=B_D(T)$ with the notation used in \cite{Bove}} that \eqref{TightR0} implies that \eqref{Hartree} has a unique global solution defined for all $t\ge 0$.

\begin{Thm}\lb{T-MFLim}
Let $\rho^{in}\in H^4(\bR^3)$ satisfy \eqref{NormaRhoin}, and let $u^{in}\in L^\infty(\bR^3)$ such that $\grad u^{in}\in H^4(\bR^3)$. Let $T>0$ be such that the system \eqref{EulP} has a solution $(\rho,u)$
satisfying the conditions $\rho\in L^\infty([0,T]\times\bR^3)$ and $\d_x^\a u_j\in L^\infty([0,T]\times\bR^3)$ for all $k=1,2,3$ and all $\a\in\bN^3$ with $|\a|\le 3$. 

Let $R_{\hbar}^{in}\in\cD(\fH)$ satisfy \eqref{TightR0} and \eqref{Monokin0}, while the density function $\rho_{\hbar}^{in}$ of $R_{\hbar}^{in}$ satisfies \eqref{EPrho0}. Let $t\mapsto R_{\hbar}(t)$ be the 
solution of the Cauchy problem for the Hartree equation \eqref{Hartree}.

Let $R_{\hbar,N}(t):=e^{-it\scrH_N/\hbar}R_{\hbar,N}^{in}e^{+it\scrH_N/\hbar}$ be the generalized solution of \eqref{NvN} with initial data $R_{\hbar,N}^{in}:=\left(R_{\hbar}^{in}\right)^{\otimes N}$, and let 
$R_{\hbar,N:1}(t)$ be the first marginal density of $R_{\hbar,N}(t)$. Denote by $\rho_{\hbar,N:1}(t,\cdot)$ and $J_{\hbar,N:1}(t,\cdot)$ respectively the density function and the current of $R_{\hbar,N:1}(t)$, 
and let $J_{\hbar}(t,\cdot)$ be the current of $R_{\hbar}(t)$.

Then
$$
\rho_{\hbar,N:1}(t,\cdot)-\rho_{\hbar}(t,\cdot)\to 0\text{ and }J_{\hbar,N:1}(t,\cdot)-J_{\hbar}(t,\cdot)\to 0
$$
for the narrow topology of signed Radon measures on $\bR^3$, while 
$$
W_{\hbar}[R_{\hbar,N:1}(t)]-W_{\hbar}[R_{\hbar}(t)]\to 0\quad\text{ in }\cS'(\bR^3)
$$
for each $t\in[0,T]$ as $\frac1N+\hbar\to 0$.
\end{Thm}

\smallskip
This theorem bears on the mean-field limit of the quantum $N$-particle dynamics leading to the Hartree equation in the large $N$ limit and for a repulsive Coulomb interaction. For $\hbar>0$ fixed, this limit
has been proved first in \cite{ErdosYau} and later in \cite{RodSchlein} with a bound on the convergence rate (see also \cite{Pickl}). Theorem \ref{T-MFLim} shows that this mean-field limit, i.e. the large $N$
limit, is uniform as $\hbar\to 0$, i.e. in the semiclassical regime, for nearly monokinetic states. The singularity of the Coulomb potential at the origin is a source of significant mathematical difficulties --- at the
time of this writing, the mean-field limit of the classical $N$-particle dynamics leading to the Vlasov-Poisson system remains an open problem. See the introduction for more details and more references on
these issues.

\smallskip
Theorem \ref{T-MFLim} is a straightforward consequence of Theorem \ref{T-MFSCLim}, and of the following proposition.

\begin{Prop}\lb{P-SCLim}
Let $\rho^{in}\in H^4(\bR^3)$ satisfy \eqref{NormaRhoin}, and let $u^{in}\in L^\infty(\bR^3)$ such that $\grad u^{in}\in H^4(\bR^3)$. Let $T>0$ be such that the system \eqref{EulP} has a solution $(\rho,u)$
satisfying the conditions $\rho\in L^\infty([0,T]\times\bR^3)$ and $\d_x^\a u_j\in L^\infty([0,T]\times\bR^3)$ for all $k=1,2,3$ and all $\a\in\bN^3$ with $|\a|\le 3$. 

Let $R_{\hbar}^{in}\in\cD(\fH)$ satisfy \eqref{TightR0} and \eqref{Monokin0}, while the density function $\rho_{\hbar}^{in}$ of $R_{\hbar}^{in}$ satisfies \eqref{EPrho0}. Let $t\mapsto R_{\hbar}(t)$ be the 
solution of the Cauchy problem for the Hartree equation \eqref{Hartree}. Then, for all $t\in[0,T]$, one has
$$
\rho_{\hbar}(t,\cdot)\to\rho(t,\cdot)\text{ and }J_{\hbar}(t,\cdot)\to\rho u(t,\cdot)
$$
for the narrow topology of signed Radon measures on $\bR^3$ as $\hbar\to 0$. Moreover
$$
\Tr_\fH(R_{\hbar}(t)^{1/2}|\hbar D_x-u(t,x)|^2R_{\hbar}(t)^{1/2})\to 0
$$
and
$$
W_{\hbar}[R_{\hbar}(t)]\to\rho(t,x)\de(\xi-u(t,x))\quad\text{ in }\cS'(\bR^3)
$$
for each $t\in[0,T]$ as $\hbar\to 0$.
\end{Prop}

\smallskip
Proposition \ref{P-SCLim} proves that the classical limit of the Hartree equation with repulsive Coulomb interaction is described by the pressureless Euler-Poisson system for nearly monokinetic states.
This result is analogous to the derivation of the Vlasov-Poisson equation as the classical limit of the Hartree equation with repulsive Coulomb interaction in Theorem IV.5 of \cite{LionsPaul}. However,
the quantum states considered in Proposition \ref{P-SCLim} differ considerably from those in Theorem IV.5 of \cite{LionsPaul}. Indeed, in the situation considered in Theorem IV.5 of \cite{LionsPaul}, 
the family $W_{\hbar}[R_{\hbar}]$ is bounded in $L^\infty([0,T];L^2(\bR^6))$, which is incompatible with the monokinetic regime considered here.

\bigskip
The proof of this result occupies remaining part of this paper, whose outline is as follows. In section \ref{S-PropaEstim}, we introduce a functional related to the modulated energy used in \cite{DS}, and we 
compute its evolution under the $N$-particle quantum dynamics \eqref{NSchro} and the Euler-Poisson flow \eqref{EulP}. The analogy between our functional and the one used in \cite{DS} will be explained
later, in Lemma \ref{L-FcF}. Using Serfaty's main inequality in \cite{DS} leads ultimately to a control of our functional by a Gronwall inequality: this is one key step in our analysis. In section \ref{S-ProofT}, 
we explain how the bound on our functional obtained in the previous section implies all the convergence statements of Theorem \ref{T-MFSCLim}, i.e. \eqref{Cvrhot}, \eqref{CvJt}, \eqref{Monokint} and 
\eqref{Wignert}. Along with the fundamental theory of Wigner measures presented in \cite{LionsPaul}, the key ingredient in this section is a variant of the Fefferman-de la Llave representation \cite{FdlL} 
of the Coulomb potential, which provides a convenient truncation of the modulated potential energy. The proofs of Theorem \ref{T-MFLim} and Proposition \ref{P-SCLim} belongs to the last section of this 
paper (section \ref{S-SCMF}).


\section{The Propagation Estimate}\lb{S-PropaEstim}


The modulated energy used in \cite{DS} involves the sum of the fluctuation of kinetic energy of the $N$-particle phase space empirical measure centered at the velocity field $u$, and of the potential energy
of the difference between the $N$-particle empirical density and the density $\rho$ (where $(\rho,u)$ is the solution of \eqref{EulP}). In the quantum setting, there is no analogue of this quantity, unless one 
is willing to use the formalism of \cite{FGTPaulEmpir}.

For this reason, we shall consider a seemingly different --- although intimately related (see Lemma \ref{L-FcF} below) --- functional. Set
$$
\ba
\cE[R_{\hbar,N},\rho,u](t):=\Tr_\fH(R_{\hbar,N:1}(t)^{1/2}|\hbar D_x-u(t,x)|^2R_{\hbar,N:1}(t)^{1/2})\,,
\\
+\iint_{\bR^6}V(x_1-x_2)\tfrac{N-1}N\rho_{\hbar,N:2}(t,x_1,x_2)dx_1dx_2+\int_{\bR^3}\rho(t,x_1)(V\star_x\rho)(t,x_1)dx_1
\\
-2\int_{\bR^3}\rho_{\hbar,N:1}(t,x_1)(V\star_x\rho)(t,x_1)dx_1&\,.
\ea
$$
where $\rho_{\hbar,N:1}(t,\cdot)$ is the density function of the first marginal $R_{\hbar,N:1}(t)$, while $\rho_{\hbar,N:2}(t,\cdot)$ is the density function of the second marginal $R_{\hbar,N:2}(t)$. 


\subsection{Step 1: energy conservation}


We begin with a computation which identifies $\cE[R_{\hbar,N},0,0](t)$ as (twice) the energy per particle in the system, and therefore justifies referring to $\cE[R_{\hbar,N},\rho,u]$ as a ``modulated energy''.

\begin{Lem}\lb{L-EnergyCons}
Let $R_{\hbar,N}^{in}\in\cD_s(\fH_N)$ satisfy 
$$
\Tr_{\fH_N}\left((R_{\hbar,N}^{in})^{1/2}\left(I+\sum_{n=1}^N-\tfrac12\hbar^2\Dlt^{(n)}\right)^2(R_{\hbar,N}^{in})^{1/2}\right)<\infty\,,
$$
where $\Dlt^{(n)}$ designates the Laplacian acting on the $n$-th factor\footnote{In other words, for each $\Psi\in C^2((\bR^3)^N)$, one sets $$(\Dlt^{(n)}\Psi)(x_1,\ldots,x_N):=\Dlt_{x_n}\Psi_N(x_1,\ldots,x_N)\,.$$}
in $(\bR^3)^N$, i.e. on the position variable of the $n$-th particle. Let $R_{\hbar,N}(t):=e^{-it\scrH_N/\hbar}R_{\hbar,N}^{in}e^{+it\scrH_N/\hbar}$. Then
$$
\cE[R_{\hbar,N},0,0](t)=\cE[R_{\hbar,N},0,0](0)\quad\text{ for all }t\ge 0\,.
$$
\end{Lem}

\begin{proof}
Let $\psi_m$ for $m\ge 1$ be a complete system of eigenfunctions of $R_{\hbar,N}^{in}$, with $R_{\hbar,N}^{in}\psi_m=\l_m\psi_m$. Our assumption implies that
$$
\ba
\|\scrH_N R_{\hbar,N}(t)^{1/2}\|_2^2=&\sum_{m\ge 1}\l_m\|e^{-it\scrH_N/\hbar}\scrH_N\psi_m\|_{\fH_N}^2
\\
=&\sum_{m\ge 1}\l_m\|\scrH_N\psi_m\|_{\fH_N}^2=\|\scrH_N(R^{in}_{\hbar,N})^{1/2}\|_2^2
\\
\le &C_N\sum_{m\ge 1}\l_m\left\|\left(I+\sum_{n=1}^N-\tfrac12\hbar^2\Dlt^{(n)}\right)\psi_m\right\|_{\fH_N}^2
\\
= &C_N\Tr_{\fH_N}\left((R_{\hbar,N}^{in})^{1/2}\left(I+\sum_{n=1}^N-\tfrac12\hbar^2\Dlt^{(n)}\right)^2(R_{\hbar,N}^{in})^{1/2}\right)<\infty\,.
\ea
$$
The first equality follows from the definition of a Hilbert-Schmidt operator, the inequality from the fact that $\Dom(\scrH_N)\subset H^2((\bR^3)^N)$, and the last equality from the definition of the trace.
Since we have assumed that the last right hand side above is finite, the inequality shows that $\l_m>0\implies\psi_m\in\Dom(\scrH_N)$, and the second equality follows from the fact that $e^{-it\scrH_N/\hbar}$
is unitary.

By the Cauchy-Schwarz inequality $R_{\hbar,N}(t)^{1/2}\scrH_N R_{\hbar,N}(t)^{1/2}\in\cL^1(\fH_N)$ and
\be\lb{ConstEnerg}
\ba
\Tr_{\fH_N}(R_{\hbar,N}(t)^{1/2}\scrH_N R_{\hbar,N}(t)^{1/2})=\sum_{m\ge 1}\l_m\la e^{-it\scrH_N/\hbar}\psi_m|e^{-it\scrH_N/\hbar}\scrH_N\psi_m\ra
\\
=\sum_{m\ge 1}\l_m\la\psi_m|\scrH_N\psi_m\ra=\Tr_{\fH_N}((R_{\hbar,N}^{in})^{1/2}\scrH_N (R_{\hbar,N}^{in})^{1/2})&\,.
\ea
\ee
Let $S_{\hbar,N}(t,X_N,Y_N)$ be the integral kernel of $R_{\hbar,N}(t)^{1/2}$; then
$$
\ba
\Tr_{\fH_N}(R_{\hbar,N}(t)^{1/2}\scrH_N R_{\hbar,N}(t)^{1/2})
\\
=\sum_{j=1}^N\iint_{\bR^{6N}}\overline{S_{\hbar,N}(t,X_N,Y_N)}(-\tfrac12\hbar^2\Dlt_{x_j})S_{\hbar,N}(t,X_N,Y_N)dY_NdX_N
\\
+\tfrac1N\sum_{1\le j<k\le N}\iint_{\bR^{6N}}\overline{S_{\hbar,N}(t,X_N,Y_N)}V(x_j-x_k)S_{\hbar,N}(t,X_N,Y_N)dY_NdX_N
\\
=\tfrac12N\iint_{\bR^{6N}}|\hbar\grad_{x_1}S_{\hbar,N}(t,X_N,Y_N)|^2dY_NdX_N
\\
+\tfrac12(N-1)\int_{\bR^{3N}}V(x_1-x_2)\rho_{\hbar,N}(t,X_N)dX_N&\,,,
\ea
$$
since $R_{\hbar,N}(t)\in\cD_s(\fH_N)$ for all $t\in\bR$. 

In terms of the first and second marginals of $R_{\hbar,N}(t)$, one finds that
$$
\ba
\Tr_{\fH_N}(R_{\hbar,N}(t)^{1/2}\scrH_N R_{\hbar,N}(t)^{1/2})
\\
=\tfrac12N\Tr_{\fH}(R_{\hbar,N:1}(t)^{1/2}|\hbar D_x|^2R_{\hbar,N:1}(t)^{1/2})
\\
+\tfrac12(N-1)\int_{\bR^{3N}}V(x_1-x_2)\rho_{\hbar,N}(t,X_N)dX_N=\tfrac12N\cE[R_{\hbar,N},0,0](t)&\,.
\ea
$$
Therefore, \eqref{ConstEnerg} implies that $\cE[R_{\hbar,N},0,0](t)=\cE[R_{\hbar,N},0,0](0)$ for each $t\ge 0$.
\end{proof}


\subsection{Step 2: evolution of the cross-term}


Throughout this paper, we use Einstein's convention of summing over repeated indices only for components of the space variable $x\in\bR^3$ and of $\grad_x$ or $D_x$, as well as for components of the 
velocity field $u\equiv u(t,x)\in\bR^3$, but never for other indices, such as particle labels.

\begin{Lem}\lb{L-CrossTerm}
Let $R_{\hbar,N}^{in}\in\cD_s(\fH_N)$ satisfy 
$$
\Tr_{\fH_N}\left((R_{\hbar,N}^{in})^{1/2}\left(I+\sum_{n=1}^N-\tfrac12\hbar^2\Dlt^{(n)}\right)^2(R_{\hbar,N}^{in})^{1/2}\right)<\infty\,,
$$
and let $R_{\hbar,N}(t):=e^{-it\scrH_N/\hbar}R_{\hbar,N}^{in}e^{+it\scrH_N/\hbar}$. Then, for each $t\ge 0$, one has
$$
\ba
\Tr_\fH(R_{\hbar,N:1}(t)^{\frac12}(u_j\vee\hbar D_j)R_{\hbar,N:1}(t)^{\frac12})
\\
-\Tr_\fH((R^{in}_{\hbar,N:1})^{\frac12}(u_j\vee\hbar D_j)(R^{in}_{\hbar,N:1})^{\frac12})
\\
=-\tfrac{N-1}N\int_0^t\iint_{\bR^6}(u_j(s,x_1)-u_j(s,x_2))\d_jV(x_1-x_2)\rho_{\hbar,N:2}(s,x_1,x_2)dx_1dx_2ds
\\
+\int_0^t\Tr((\tfrac12\hbar D_k\vee\d_ku_j-u_k\d_ku_j-\d_jV\star_x\rho)\vee\hbar D_j)R_{\hbar,N:1}(s))ds&\,.
\ea
$$
In the formula above, $u_j\equiv u_j(t,x)$ designates the $j$-th component of the vector $u(t,x)\in\bR^3$, while $\d_j$ (resp. $D_j$) designates the $j$-th component of the vector $\grad_x$ (resp. $-i\grad_x$),
for $j=1,2,3$. 
\end{Lem}

\begin{proof}
Since $\hbar D_j R_{\hbar,N:1}(t)^{1/2}\in\cL^2(\fH)$, one has 
$$
u_j\vee\hbar D_jR_{\hbar,N:1}(t)=2u_j\hbar D_jR_{\hbar,N:1}(t)+[\hbar D_j,u_j]R_{\hbar,N:1}(t)\in\cL^1(\fH)\,,
$$ 
and
$$
\ba
\Tr_\fH(R_{\hbar,N:1}(t)^{1/2}(u_j\vee\hbar D_j)R_{\hbar,N:1}(t)^{1/2})=\Tr_\fH((u_j\vee\hbar D_j)R_{\hbar,N:1}(t))
\\
=\Tr(((u_j\vee\hbar D_j)\otimes I_{\fH_{N-1}})R_{\hbar,N}(t))&\,.
\ea
$$
Write
$$
R_{\hbar,N}(t)=e^{-it\scrK_N/\hbar}\left(R_{\hbar,N}^{in}+\tfrac1{i\hbar}S_{\hbar,N}(t)\right)e^{it\scrK_N/\hbar}
$$
with
$$
\scrK_N:=\sum_{n=1}^N-\tfrac12\hbar^2\Dlt^{(n)}\,,
$$
and
$$
S_{\hbar,N}(t):=\tfrac1N\sum_{1\le k<l\le N}\int_0^te^{is\scrK_N/\hbar}[V_{kl},R_{\hbar,N}(s)]e^{-is\scrK_N/\hbar}ds\,.
$$
For each $\eps>0$, set $D^\eps_x:=Z_\eps D_x$, where $Z_\eps\psi(x):=\zeta_\eps\star\psi(x)$ for each $\psi\in\fH$, with $\zeta_\eps$ being a (real-valued), nonincreasing radial mollifier --- so that in 
particular $Z_\eps=Z_\eps^*$. Observe that
$$
\ba
\Tr_{\fH_N}(((u_j\vee\hbar D^\eps_j)\otimes I_{\fH_{N-1}})R_{\hbar,N}(t))
\\
=\Tr_{\fH_N}\left(e^{it\scrK_N/\hbar}((u_j\vee\hbar D^\eps_j)\otimes I_{\fH_{N-1}})e^{-it\scrK_N/\hbar}\left(R_{\hbar,N}^{in}+\tfrac1{i\hbar}S_{\hbar,N}(t)\right)\right)&\,;
\ea
$$
since $u\in C([0,T];W^{1,\infty}(\bR^3))$ is a solution of \eqref{EulP}, so that $\d_tu\in L^\infty([0,T]\times\bR^3)$, while $\scrH_NR_{\hbar,N}(t)$ and $R_{\hbar,N}(t)\scrH_N$ belong to $\cL^1(\fH_N)$ 
for all $t\in\bR$, one has
$$
\ba
\frac{d}{dt}\Tr_{\fH_N}(((u_j\vee\hbar D^\eps_j)\otimes I_{\fH_{N-1}})R_{\hbar,N}(t))
\\
=\Tr_{\fH_N}(((\d_tu_j\vee\hbar D^\eps_j)\otimes I_{\fH_{N-1}})R_{\hbar,N}(t))
\\
+\Tr_{\fH_N}\left(e^{it\scrK_N/\hbar}(\tfrac{i}\hbar[-\tfrac12\hbar^2\Dlt,u_j\vee\hbar D^\eps_j]\otimes I_{\fH_{N-1}})e^{-it\scrK_N/\hbar}\left(R_{\hbar,N}^{in}+\tfrac1{i\hbar}S_{\hbar,N}(t)\right)\right)
\\
+\tfrac1{i\hbar}\Tr_{\fH_N}\left(e^{it\scrK_N/\hbar}((u_j\vee\hbar D^\eps_j)\otimes _{\fH_{N-1}})e^{-it\scrK_N/\hbar}\frac{d}{dt}S_{\hbar,N}(t)\right)
\\
=\Tr_{\fH_N}(((\d_tu_j\vee\hbar D^\eps_j)\otimes I_{\fH_{N-1}})R_{\hbar,N}(t))
\\
+\Tr_{\fH_N}((\tfrac{i}\hbar[-\tfrac12\hbar^2\Dlt,u_j\vee\hbar D^\eps_j]\otimes I_{\fH_{N-1}})R_{\hbar,N}(t))
\\
+\tfrac1{i\hbar}\Tr_{\fH_N}\left(((u_j\vee\hbar D^\eps_j)\otimes I_{\fH_{N-1}})\tfrac1N\sum_{1\le k<l\le N}[V_{kl},R_{\hbar,N}(t)]\right)&\,.
\ea
$$

We compute successively
$$
\d_tu_j\vee\hbar D^\eps_j=-(u_k\d_ku_j+\d_jV\star_x\rho)\vee\hbar D^\eps_j\,,
$$
and
$$
\tfrac{i}\hbar[-\tfrac12\hbar^2\Dlt,u_j\vee\hbar D^\eps_j]=\tfrac{i}\hbar[-\tfrac12\hbar^2\Dlt,u_j]\vee\hbar D^\eps_j=\tfrac12(\hbar D_k\vee\d_ku_j)\vee\hbar D^\eps_j\,.
$$

Next we observe that $l>k>1\implies[V_{kl},(u_j\vee\hbar D^\eps_j)\otimes I_{\fH_{N-1}}]=0$, so that
$$
\ba
\Tr_{\fH_N}\left(((u_j\vee\hbar D^\eps_j)\otimes I_{\fH_{N-1}})\tfrac1N\sum_{1\le k<l\le N}[V_{kl},R_{\hbar,N}(t)]\right)
\\
=\Tr_{\fH_N}\left(((u_j\vee\hbar D^\eps_j)\otimes I_{\fH_{N-1}})\tfrac1N\sum_{l=2}^N[V_{1l},R_{\hbar,N}(t)]\right)
\\
=\Tr_{\fH_2}(((u_j\vee\hbar D^\eps_j)\otimes I)\tfrac{N-1}N[V_{12},R_{\hbar,N:2}(t)])
\\
=\tfrac12\Tr_{\fH_2}(((u_j\vee\hbar D^\eps_j)\otimes I+I\otimes(u_j\vee\hbar D^\eps_j))\tfrac{N-1}N[V_{12},R_{\hbar,N:2}(t)])
\\
=-\Tr_{\fH_2}(\tfrac12[V_{12},(u_j\vee\hbar D^\eps_j)\otimes I+I\otimes(u_j\vee\hbar D^\eps_j)]\tfrac{N-1}NR_{\hbar,N:2}(t))&\,.
\ea
$$
The second equality follows from the fact that $R_{\hbar,N}(t)\in\cD_s(\fH_N)$, and the third equality from the symmetry of $[V_{12},R_{\hbar,N:2}(t)]$ since $V$ is an even function and 
$R_{\hbar,N:2}(t)\in\cD_s(\fH_2)$. Then
$$
\ba
[V_{12},(u_j\vee\hbar D^\eps_j)\otimes I+I\otimes(u_j\vee\hbar D^\eps_j)]
\\
=(u_j\otimes I)\vee[V_{12},\hbar D^\eps_j\otimes I]+(I\otimes u_j)\vee[V_{12},I\otimes\hbar D^\eps_j]&\,,
\ea
$$
and since $V$ is even while $Z_\eps:=(I-\eps\Dlt)^{-1/2}$ is a self-adjoint convolution operator, 
$$
\ba
{}[V_{12},\hbar D^\eps_j\otimes I]=[V^\eps_{12},\hbar D_j\otimes I]=-\hbar(D_jV^\eps)_{12}\,,
\\
[V_{12},I\otimes\hbar D^\eps_j]=[V^\eps_{12},I\otimes\hbar D_j]=+\hbar(D_jV^\eps)_{12}\,.
\ea
$$
Finally, one arrives at the formula
$$
\ba
\tfrac12[V_{12},(u_j\vee\hbar D^\eps_j)\otimes I+I\otimes(u_j\vee\hbar D^\eps_j)]
\\
=i\hbar(u_j(t,x_1)-u_j(t,x_2))\d_jV^\eps(x_1-x_2)&\,.
\ea
$$
Summarizing, we have proved that
$$
\ba
\frac{d}{dt}\Tr(R_{\hbar,N:1}(t)^{1/2}(u_j\vee\hbar D^\eps_j)R_{\hbar,N:1}(t)^{1/2})
\\
=-\Tr((u_k\d_ku_j+\d_jV\star_x\rho)\vee\hbar D^\eps_j)R_{\hbar,N:1}(t))
\\
+\Tr((\tfrac12(\hbar D_k\vee\d_ku_j)\vee\hbar D^\eps_j)R_{\hbar,N:1}(t))
\\
-\Tr((u_j(t,x_1)-u_j(t,x_2))\d_jV^\eps(x_1-x_2)\tfrac{N-1}NR_{\hbar,N:2}(t))&\,.
\ea
$$

At this point, we integrate both sides of this equality over $[0,t]$ and let $\eps\to 0$ in the r.h.s. of the equality. Since $(\rho,u)\in L^\infty([0,T];L^\infty(\bR^3)\times W^{1,\infty}(\bR^3)^3)$ while 
$t\mapsto\scrH_NR_{\hbar,N}(t)$ and $t\mapsto R_{\hbar,N}(t)\scrH_N$ belong to $C(\bR;\cL^1(\fH_N))$,
$$
\ba
\Tr((u_k\d_ku_j+\d_jV\star_x\rho)\vee\hbar D^\eps_j)R_{\hbar,N:1}(t))
\\
\to\Tr((u_k\d_ku_j+\d_jV\star_x\rho)\vee\hbar D_j)R_{\hbar,N:1}(t))
\\
\Tr((\tfrac12(\hbar D_k\vee\d_ku_j)\vee\hbar D^\eps_j)R_{\hbar,N:1}(t))
\\
\to\Tr((\tfrac12(\hbar D_k\vee\d_ku_j)\vee\hbar D_j)R_{\hbar,N:1}(t))
\ea
$$
uniformly on $[0,T]$ as $\eps\to 0$. As for the last term,
$$
\ba
\Tr((u_j(t,x_1)-u_j(t,x_2))\d_jV^\eps(x_1-x_2)R_{\hbar,N:2}(t))
\\
=\iint_{\bR^6}(u_j(t,x_1)-u_j(t,x_2))\d_jV^\eps(x_1-x_2)\rho_{\hbar,N:2}(t,x_1,x_2)dx_1dx_2&\,,
\ea
$$
and we already know from the conservation of energy that
$$
\iint_{\bR^6}V(x_1-x_2)\rho_{\hbar,N:2}(t,x_1,x_2)dx_1dx_2<\infty\,.
$$
Since
\be\lb{BoundGradVeps}
|(u_j(t,x_1)-u_j(t,x_2))\d_jV^\eps(x_1-x_2)|\le C\|\grad_xu\|_{L^\infty([0,T]\times\bR^3)}V(x_1-x_2)
\ee
for all $t\in[0,T]$ and all $x_1\not=x_2\in\bR^3$ (see below a proof of \eqref{BoundGradVeps}), we conclude by dominated convergence that
$$
\ba
\int_0^t\Tr_{\fH_2}((u_j(s,x_1)-u_j(s,x_2))\d_jV^\eps(x_1-x_2)R_{\hbar,N:2}(s))ds
\\
=\int_0^t\iint_{\bR^6}(u_j(s,x_1)-u_j(s,x_2))\d_jV^\eps(x_1-x_2)\rho_{\hbar,N:2}(s,x_1,x_2)dx_1dx_2ds
\\
\to\int_0^t\iint_{\bR^6}(u_j(s,x_1)-u_j(s,x_2))\d_jV(x_1-x_2)\rho_{\hbar,N:2}(s,x_1,x_2)dx_1dx_2ds
\ea
$$
as $\eps\to 0$.
\end{proof}

\begin{proof}[Proof of \eqref{BoundGradVeps}]
For each $\eps>0$ and each $x\in\bR^3$, one has
$$
4\pi|\grad V^\eps(x)|\le\int_{\bR^3}\frac{\zeta_\eps(z)dz}{|x-z|^2}\,.
$$
Since $\zeta_\eps(z):=\eps^{-3}\zeta(z/\eps)$, where $\zeta$ is a nonnegative, radial nonincreasing function on $\bR^3$ such that
$$
\int_{\bR^3}\zeta(z)dz=1\quad\text{ and }\quad\Supp(\zeta)\subset B(0,1)\,.
$$
If $\eps\le|x|/2$, then
$$
\int_{\bR^3}\frac{\zeta_\eps(z)dz}{|x-z|^2}\le\int_{|z|<|x|/2}\frac{\zeta_\eps(z)dz}{(|x|-|z|)^2}\le\int_{\bR^3}\frac{\zeta_\eps(z)dz}{(|x|-\frac12|x|)^2}=\frac4{|x|^2}\int_{\bR^3}\zeta_\eps(z)dz=\frac4{|x|^2}\,.
$$
If $\eps>|x|/2$, since $z\mapsto|z|^{-2}$ is the symmetric-decreasing rearrangement of $z\mapsto|x-z|^{-2}$ and since $\zeta_\eps$ is it own symetric-decreasing rearrangement, 
$$
\int_{\bR^3}\frac{\zeta_\eps(z)dz}{|x-z|^2}\le\int_{\bR^3}\frac{\zeta_\eps(z)dz}{|z|^2}=\frac1{\eps^2}\int_{|y|\le 1}\frac{\zeta(y)dy}{|y|^2}\le\frac4{|x|^2}\int_{|y|\le 1}\frac{\zeta(y)dy}{|y|^2}\,,
$$
by Theorem 3.4 in \cite{LiebLoss}. Hence
$$
|x||\grad V^\eps(x)|\le CV(x)\,,\qquad\text{ with }C:=4\int_{\bR^3}\frac{\zeta(y)dy}{|y|^2}\,.
$$
\end{proof}


\subsection{Step 3: consequences of the local conservation of mass}


\begin{Lem}\lb{L-MassCons}
Assume that $R_{\hbar,N}^{in}\in\cD_s(\fH_N)$ satisfies 
$$
\Tr_{\fH_N}\left((R_{\hbar,N}^{in})^{1/2}\left(I+\sum_{j=1}^N-\tfrac12\hbar^2\Dlt_j\right)^2(R_{\hbar,N}^{in})^{1/2}\right)<\infty\,.
$$
and let $R_{\hbar,N}(t):=e^{-it\scrH_N/\hbar}R_{\hbar,N}^{in}e^{+it\scrH_N/\hbar}$. Let $R_{\hbar,N:1}(t)$ be the first marginal of $R_{\hbar,N}(t)$, and let $\rho_{\hbar,N:1}(t,\cdot)$ and $J_{\hbar,N:1}(t,\cdot)$
be the density function and the current of $R_{\hbar,N}(t)$. Then
$$
\d_t\rho_{\hbar,N:1}(t,x_1)+\Div_{x_1}J_{\hbar,N:1}(t,dx_1)=0\,,\quad\text{ in }\cD'((0,T)\times\bR^3)\,.
$$
\end{Lem}

\begin{proof}
Let $a\in C^1_b(\bR^3)$. Since $R_{\hbar,N}\in C(\bR;\Dom(\scrH_N))\cap C^1(\bR;\fH_N)$, one has
$$
\ba
\frac{d}{dt}\int_{\bR^3}a(x_1)\rho_{\hbar,N:1}(t,x_1)dx_1=&\frac{d}{dt}\Tr_{\fH}(aR_{\hbar,N:1}(t))
\\
=&\frac{d}{dt}\Tr_{\fH_N}((a\otimes I_{\fH_{N-1}})R_{\hbar,N}(t))
\\
=&-\tfrac1{i\hbar}\sum_{n=1}^N\Tr_{\fH_N}([-\tfrac12\hbar^2\Dlt^{(n)},a\otimes I_{\fH_{N-1}}]R_{\hbar,N}(t))
\\
=&-\tfrac1{i\hbar}\Tr_{\fH_N}(([-\tfrac12\hbar^2\Dlt,a]\otimes I_{\fH_{N-1}})R_{\hbar,N}(t))
\\
=&-\tfrac1{i\hbar}\Tr_{\fH}([-\tfrac12\hbar^2\Dlt,a]R_{\hbar,N:1}(t))
\\
=&\Tr_{\fH}(\tfrac12\hbar D_k\vee[\d_k,a]R_{\hbar,N:1}(t))
\\
=&\Tr_{\fH}(\tfrac12(\hbar D_k\vee\d_ka)R_{\hbar,N:1}(t))
\\
=&\int_{\bR^3}\d_ka(x_1)J^k_{\hbar,N:1}(t,dx_1)\,.
\ea
$$
\end{proof}

\smallskip
Using Lemma \ref{L-MassCons} with $a=|u(t,\cdot)|^2$ implies that
\be\lb{Tru2R1}
\ba
\frac{d}{dt}\Tr_{\fH}(|u(t,x)|^2R_{\hbar,N:1}(t))=\frac{d}{dt}\int_{\bR^3}|u(t,x)|^2\rho_{\hbar,N:1}(t,x_1)dx_1
\\
=-2\int_{\bR^3}(u_j(t,x_1)(u_k\d_ku_j(t,x_1)+\d_jV\star_x\rho(t,x_1)))\rho_{\hbar,N:1}(t,x_1)dx_1
\\
+2\int_{\bR^3}(u_j\d_ku_j)(t,x_1)J_{\hbar,N:1}^k(t,dx_1)&\,.
\ea
\ee

Next we use Lemma \ref{L-MassCons} with $a=V\star_x\rho(t,\cdot)$ to compute
\be\lb{dtEnerPot}
\ba
\frac{d}{dt}\int_{\bR^3}\!(V\!\star_x\!\rho)(t,x)(\rho(t,x)\!-\!2\rho_{\hbar,N:1}(t,x))dx
\\
=2\!\int_{\bR^3}\!\rho(t,x)u_j(t,x)\d_jV\!\star_x\!\rho(t,x)dx
\\
+2\!\int_{\bR^3}\!\d_jV\!\star_x\!(\rho u_j)(t,x)\rho_{\hbar,N:1}(t,x)dx
\\
-2\!\int_{\bR^3}\!\d_k(V\!\star_x\!\rho)(t,x)J^k_{\hbar,N:1}(t,dx)&\,.
\ea
\ee

\bigskip
Putting together the energy conservation (Lemma \ref{L-EnergyCons}), the computation of the cross term (Lemma \ref{L-CrossTerm}), \eqref{Tru2R1} and \eqref{dtEnerPot}, we arrive at the following formula
for the evolution of the modulated energy $\cE[R_{\hbar,N},\rho,u]$:
$$
\ba
\cE[R_{\hbar,N},\rho,u](t)-\cE[R_{\hbar,N},\rho,u](0)
\\
=\tfrac{N-1}N\int_0^t\iint_{\bR^6}(u_j(s,x_1)-u_j(s,x_2))\d_jV(x_1-x_2)\rho_{\hbar,N:2}(s,x_1,x_2)dx_1dx_2ds
\\
-\!\int_0^t\!\Tr\left(R_{\hbar,N:1}(s)^{\frac12}(\tfrac12\hbar D_k\vee\d_ku_j\!-\!u_k\d_ku_j\!-\!\d_jV\star_x\rho)\vee\hbar D_j)R_{\hbar,N:1}(s)^{\frac12}\right)ds
\\
-2\int_0^t\int_{\bR^3}(u_j(s,x_1)(u_k\d_ku_j(s,x_1)+\d_jV\star_x\rho(s,x_1)))\rho_{\hbar,N:1}(s,x_1)dx_1ds
\\
+2\int_0^t\int_{\bR^3}(u_j\d_ku_j)(s,x_1)J_{\hbar,N:1}^k(s,dx_1)ds
\\
+2\!\int_0^t\int_{\bR^3}\!\rho(s,x)u_j(s,x)\d_jV\!\star_x\!\rho(s,x)dxds
\\
-2\!\int_0^t\int_{\bR^3}\!\rho(s,x)u_j(s,x)\d_jV\!\star_x\!\rho_{\hbar,N:1}(s,x)dxds
\\
-2\!\int_0^t\int_{\bR^3}\!\d_k(V\!\star_x\!\rho)(s,x)J^k_{\hbar,N:1}(s,dx)ds&\,.
\ea
$$
The various terms on the right hand side can be grouped in a slightly more convenient manner:
$$
\ba
\cE[R_{\hbar,N},\rho,u](t)-\cE[R_{\hbar,N},\rho,u](0)
\\
=\tfrac{N-1}N\int_0^t\iint_{\bR^6}(u_j(s,x_1)-u_j(s,x_2))\d_jV(x_1-x_2)\rho_{\hbar,N:2}(s,x_1,x_2)dx_1dx_2ds
\\
+2\int_0^t\!\int_{\bR^3}\!\rho u_j\d_jV\!\star_x\!\rho(s,x)dxds+2\int_0^t\!\int_{\bR^3}\!(\d_jV\!\star_x(\rho u_j))\rho_{\hbar,N:1}(s,x)dxds
\\
-2\int_0^t\int_{\bR^3}(u_j\d_jV\star_x\rho)\rho_{\hbar,N:1}(s,x)dxds
\\
-\int_0^t\Tr(R_{\hbar,N:1}(s)^{\frac12}((\tfrac12\hbar D_k\vee\d_ku_j)\vee\hbar D_j+2u_ju_k\d_ku_j)R_{\hbar,N:1}(s)^{\frac12})ds
\\
+\int_0^t\Tr(((u_k\d_ku_j)\vee\hbar D_j+(u_j\d_ku_j)\vee\hbar D_k)R_{\hbar,N:1}(s))ds&\,,
\ea
$$
since
$$
\ba
\Tr(((\d_jV\star_x\rho(s,\cdot))\vee\hbar D_j)R_{\hbar,N:1}(s))=2\!\int_{\bR^3}\!\d_k(V\!\star_x\!\rho)(s,x)J^k_{\hbar,N:1}(s,dx)&\,,
\\
\Tr(((u_j\d_ku_j)(s,\cdot)\vee\hbar D_k)R_{\hbar,N:1}(s))=2\int_{\bR^3}(u_j\d_ku_j)(s,x)J^k_{\hbar,N:1}(s,dx)&\,.
\ea
$$
Observe that
$$
\ba
2\int_{\bR^3}\!\rho u_j\d_jV\!\star_x\!\rho(s,x)dx=2\iint_{\bR^6}u_j(s,x_1)\d_jV(x_1-x_2)\rho(s,x_1)\rho(s,x_2)dx_1dx_2
\\
=\iint_{\bR^6}(u_j(s,x_1)-u_j(s,x_2)\d_jV(x_1-x_2)\rho(s,x_1)\rho(s,x_2)dx_1dx_2
\ea
$$
by symmetry (since $\d_jV$ is an odd function), while
$$
\ba
2\int_{\bR^3}\!(\d_jV\!\star_x(\rho u_j))\rho_{\hbar,N:1}(s,x)dx-2\int_{\bR^3}(u_j\d_jV\star_x\rho)\rho_{\hbar,N:1}(s,x)dx
\\
=2\iint_{\bR^6}(u_j(s,x_2)-u_j(s,x_1))\d_jV(x_1-x_2)\rho_{\hbar,N:1}(s,x_1)\rho(s,x_2)dx_1dx_2&\,.
\ea
$$

On the other hand
$$
\ba
(\tfrac12\hbar D_k\vee\d_ku_j)\vee\hbar D_j+2u_ju_k\d_ku_j)-(u_k\d_ku_j)\vee\hbar D_j-(u_j\d_ku_j)\vee\hbar D_k
\\
=\tfrac12\hbar D_k\d_ku_j\hbar D_j+\tfrac12\d_ku_j\hbar D_k\hbar D_j+\tfrac12\hbar D_j\hbar D_k\d_ku_j+\tfrac12\hbar D_j\d_ku_j\hbar D_k
\\
+2u_k\d_ku_ju_j-u_k\d_ku_j\hbar D_j-\hbar D_j u_k\d_ku_j-u_j\d_ku_j\hbar D_k-\hbar D_ku_j\d_ku_j
\\
=\tfrac12(\hbar D_k-u_k)\d_ku_j(\hbar D_j-u_j)+\tfrac12(\hbar D_j-u_j)\d_ku_j(\hbar D_k-u_k)
\\
+\tfrac12\d_ku_j(\hbar D_k-u_k)(\hbar D_j-u_j)+\tfrac12(\hbar D_j-u_j)(\hbar D_k-u_k)\d_ku_j
\\
=(\hbar D_k-u_k)\d_ku_j(\hbar D_j-u_j)+(\hbar D_j-u_j)\d_ku_j(\hbar D_k-u_k)
\\
-\tfrac12[\hbar D_k,\d_ku_j](\hbar D_j-u_j)+\tfrac12(\hbar D_j-u_j)[\hbar D_k,\d_ku_j]
\\
=2(\hbar D_k-u_k)\Si_{jk}(\hbar D_j-u_j)-\tfrac12\hbar^2\Dlt_x\Div_xu&\,,
\ea
$$
where $\Si_{jk}:=\tfrac12(\d_ku_j+\d_ju_k)$, since
$$
\ba
\tfrac12\d_ku_j(\hbar D_k)u_j+\tfrac12u_j(\hbar D_k)\d_ku_j-\tfrac12u_j\d_ku_j(\hbar D_k)-\tfrac12(\hbar D_k)u_j\d_ku_j
\\
=\tfrac12\d_ku_j[\hbar D_k,u_j]+\tfrac12[u_j,\hbar D_k]\d_ku_j=-\tfrac12i\hbar\d_ku_j\d_ku_j+\tfrac12i\hbar\d_ku_j\d_ku_j=0&\,.
\ea
$$

\bigskip
Summarizing, we have proved the following identity.

\begin{Prop}\lb{P-DtE}
Let $R_{\hbar,N}^{in}\in\cD_s(\fH_N)$ satisfy 
$$
\Tr_{\fH_N}\left((R_{\hbar,N}^{in})^{1/2}\left(I+\sum_{j=1}^N-\tfrac12\hbar^2\Dlt_j\right)^2(R_{\hbar,N}^{in})^{1/2}\right)<\infty\,,
$$
and let $R_{\hbar,N}(t):=e^{-it\scrH_N/\hbar}R_N^{in}e^{+it\scrH_N/\hbar}$. Then, for each $t\in[0,T]$, one has
$$
\ba
\cE[R_{\hbar,N},\rho,u](t)-\cE[R_{\hbar,N},\rho,u](0)
\\
=-2\int_0^t\Tr\left(R_{\hbar,N:1}(s)^{\frac12}(\hbar D_k-u_k)\Si_{jk}(\hbar D_j-u_j)R_{\hbar,N:1}(s)^{\frac12}\right)ds
\\
+\tfrac12\hbar^2\int_0^t\int_{\bR^3}(\Dlt_x\Div_xu)(s,x)\rho_{\hbar,N:1}(s,x)dxds
\\
+\int_0^t\iint_{\bR^6}(u_j(s,x_1)-u_j(s,x_2))\d_jV(x_1-x_2)\big(\tfrac{N-1}N\rho_{\hbar,N:2}(s,x_1,x_2)
\\
+\rho(s,x_1)\rho(s,x_2)-2\rho_{\hbar,N:1}(s,x_1)\rho(s,x_2)\big)dx_1dx_2ds&\,.
\ea
$$
\end{Prop}


\subsection{Step 4: the Gronwall inequality}


By the Cauchy-Schwarz inequality, one has
\be\lb{Kin'<CKin}
\ba
\left|\Tr\left(R_{\hbar,N:1}(s)^{\frac12}(\hbar D_k-u_k)\Si_{jk}(\hbar D_j-u_j)R_{\hbar,N:1}(s)^{\frac12}\right)\right|
\\
\le\|R_{\hbar,N:1}(s)^{\frac12}(\hbar D_k-u_k)\|_2\|\Si_{jk}(\hbar D_j-u_j)R_{\hbar,N:1}(s)^{\frac12}\|_2
\\
\le\sup_{1\le j,k\le 3}\|\Si_{jk}\|_{L^\infty([0,T]\times\bR^3)}\|R_{\hbar,N:1}(s)^{\frac12}(\hbar D_k-u_k)\|_2\|(\hbar D_j-u_j)R_{\hbar,N:1}(s)^{\frac12}\|_2
\\
\le\sup_{1\le j,k\le 3}\|\Si_{jk}\|_{L^\infty([0,T]\times\bR^3)}\Tr\left(R_{\hbar,N:1}(s)^{\frac12}|\hbar D-u(s,\cdot)|^2R_{\hbar,N:1}(s)^{\frac12}\right)&\,.
\ea
\ee
On the other hand
\be\lb{h2term}
\ba
\left|\int_0^t\int_{\bR^3}(\Dlt_x\Div_xu)(s,x)\rho_{\hbar,N:1}(s,x)dxds\right|
\\
\le\|\Dlt_x\Div_xu\|_{L^\infty([0,T]\times\bR^3)}\int_0^t\int_{\bR^3}\rho_{\hbar,N:1}(s,x)dxds
\\
=t\|\Dlt_x\Div_xu\|_{L^\infty([0,T]\times\bR^3)}&\,.
\ea
\ee

The remaining term, which comes from the modulated potential energy and involves the Coulomb interaction potential, is more involved. We recast this term as
\be\lb{DefcF}
\ba
\cF[R_{\hbar,N},\rho,u](t):=\iint_{\bR^6}V(x_1-x_2)\tfrac{N-1}N\rho_{\hbar,N:2}(t,x_1,x_2)dx_1dx_2
\\
+\int_{\bR^3}\rho(t,x_1)(V\star_x\rho)(t,x_1)dx_1-2\int_{\bR^3}\rho_{\hbar,N:1}(t,x_1)(V\star_x\rho)(t,x_1)dx_1
\\
=\iint_{\bR^6}V(x_1-x_2)\big(\tfrac{N-1}N\rho_{\hbar,N:2}(t,x_1,x_2)+\rho(t,x_1)\rho(t,x_2)
\\
-2\rho_{\hbar,N:1}(t,x_1)\rho(t,x_2)\big)dx_1dx_2&\,.
\ea
\ee
Likewise, we set
\be\lb{DefcF'}
\ba
\cF'[R_{\hbar,N},\rho,u](t):=\iint_{\bR^6}(u_j(t,x_1)-u_j(t,x_2))\d_jV(x_1-x_2)
\\
\times\big(\tfrac{N-1}N\rho_{\hbar,N:2}(t,x_1,x_2)+\rho(t,x_1)\rho(t,x_2)-2\rho_{\hbar,N:1}(t,x_1)\rho(t,x_2)\big)dx_1dx_2&\,.
\ea
\ee

The next lemma connects Serfaty's functional $F_N$ defined in formula (1.11) of \cite{DS} with our variant of the modulated potential energy $\cF_N$.

\begin{Lem}\lb{L-FcF}
For each $X_N=(x_1,\ldots,x_N)\in(\bR^3)^N$, set
$$
\mu_{X_N}:=\frac1N\sum_{n=1}^N\de_{x_k}\,,
$$
and, denoting by $\bD$ the diagonal of $\bR^3\times\bR^3$,
$$
F_N(X_N,\rho(t,\cdot)):=N^2\iint_{\bR^6\setminus\bD}V(x-y)(\mu_{X_N}(dx)-\rho(t,dx))(\mu_{X_N}(dy)-\rho(t,dy))\,.
$$
Likewise, set
$$
\ba
F'_N(X_N,(\rho,u)(t,\cdot)):=N^2\iint_{\bR^6\setminus\bD}&(u_j(t,x)-u_j(t,y))\d_jV(x-y)
\\
&\times(\mu_{X_N}(dx)-\rho(t,dx))(\mu_{X_N}(dy)-\rho(t,dy))\,.
\ea
$$
Then
$$
\cF[R_{\hbar,N},\rho,u](t)=\frac1{N^2}\int_{(\bR^3)^N}F_N(X_N,\rho(t,\cdot))\rho_{\hbar,N}(t,X_N)dX_N\,,
$$
and
$$
\cF'[R_{\hbar,N},\rho,u](t)=\frac1{N^2}\int_{(\bR^3)^N}F'_N(X_N,(\rho,u)(t,\cdot))\rho_{\hbar,N}(t,X_N)dX_N\,.
$$
\end{Lem}

\begin{proof}
Let $w\in C(\bR^3\times\bR^3\setminus\bD)$ satisfy 
$$
w(x,y)=w(y,x)\text{ and }|w(x,y)|\le c(1+\tfrac1{|x-y|})\qquad\text{ for each }(x,y)\in\bR^3\times\bR^3\setminus\bD
$$
for some constant $c>0$. Then
$$
\ba
\iint_{\bR^6\setminus\bD}w(x,y)(\mu_{X_N}(dx)\!-\!\rho(t,dx))(\mu_{X_N}(dy)\!-\!\rho(t,dy))
\\
=\!\!\iint_{\bR^6\setminus\bD}\!w(x,y)\mu_{X_N}(dx)\mu_{X_N}(dy)
\\
+\iint_{\bR^6}w(x,y)\rho(t,x)\rho(t,y)dxdy-2\iint_{\bR^6}w(x,y)\mu_{X_N}(dx)\rho(t,y)dy
\\
=\frac1{N^2}\sum_{1\le j\not=k\le N}w(x_i,x_j)-\frac2N\sum_{j=1}^N\int_{\bR^3}w(x_j,y)\rho(t,y)dy
\\
+\iint_{\bR^6}w(x,y)\rho(t,x)\rho(t,y)dxdy
\ea
$$
for a.e. $X_N\in(\bR^3)^N$. Since $|w(x,y)|=O(|x-y|^{-1})$ near $\bD$ and $\rho(t,\cdot)$ is continuous and bounded on $\bR^3$, the function
$$
x\mapsto\int_{\bR^3}w(x,y)\rho(t,y)dy
$$
is continuous on $\bR^3$, so that $(x,y)\mapsto w(x,y)\rho(t,x)\rho(t,y)$ is integrable on $\bR^6$. Hence
$$
\ba
\iint_{\bR^6\setminus\bD}w(x,y)\rho(t,x)\rho(t,y)dxdy=\iint_{\bR^6}w(x,y)\rho(t,x)\rho(t,y)dxdy\,,
\\
\iint_{\bR^6\setminus\bD}w(x,y)\mu_{X_N}(dx)\rho(t,y)dy=\iint_{\bR^6}w(x,y)\mu_{X_N}(dx)\rho(t,y)dy\,.
\ea
$$
The only place where removing the diagonal in the domain of integration is important is in the computation of
$$
\iint_{\bR^6\setminus\bD}w(x,y)\mu_{X_N}(dx)\mu_{X_N}(dy)=\frac1{N^2}\sum_{1\le j\not=k\le N}w(x_i,x_j)\,,
$$
provided that $X_N$ is such that $j\not=k\implies x_j\not=x_k$. This restriction obviously consists of removing from $(\bR^3)^N$ a finite union of linear varieties of codimension $3$, which is a 
Lebesgue-negligible set in $(\bR^3)^N$. Thus
$$
\ba
\int_{(\bR^3)^N}\!\!\iint_{\bR^6\setminus\bD}w(x,y)(\mu_{X_N}(dx)\!-\!\rho(t,dx))(\mu_{X_N}(dy)\!-\!\rho(t,dy))\rho_{\hbar,N}(t,X_N)dX_N
\\
=\iint_{\bR^6}w(x,y)\rho(t,x)\rho(t,y)dxdy\int_{(\bR^3)^N}\rho_{\hbar,N}(t,X_N)dX_N
\\
-\frac2N\int_{(\bR^3)^N}\left(\sum_{j=1}^N\int_{\bR^3}w(x_j,y)\rho(t,y)dy\right)\rho_{\hbar,N}(t,X_N)dX_N
\\
+\frac1{N^2}\int_{(\bR^3)^N}\sum_{1\le j\not=k\le N}w(x_i,x_j)\rho_{\hbar,N}(t,X_N)dX_N&\,.
\ea
$$
Since $\rho_N(t,\cdot)$ is a probability density on $(\bR^3)^N$, one has
$$
\ba
\iint_{\bR^6}w(x,y)\rho(t,x)\rho(t,y)dxdy\int_{\bR^3)^N}\rho_N(t,X_N)dX_N
\\
=\iint_{\bR^6}w(x,y)\rho(t,x)\rho(t,y)dxdy&\,.
\ea
$$
Since $R_{\hbar,N}(t)\in\cD_s(\fH_N)$ for each $t\in\bR$, the function $X_N\mapsto\rho_{\hbar,N}(t,X_N)$ is symmetric, so that
$$
\ba
\frac2N\int_{(\bR^3)^N}&\left(\sum_{j=1}^N\int_{\bR^3}w(x_j,y)\rho(t,y)dy\right)\rho_{\hbar,N}(t,X_N)dX_N
\\
&=2\int_{(\bR^3)^N}\left(\int_{\bR^3}w(x_1,y)\rho(t,y)dy\right)\rho_{\hbar,N}(t,X_N)dX_N
\\
&=2\int_{\bR^3}\left(\int_{\bR^3}w(x_1,x_2)\rho(t,x_2)dx_2\right)\rho_{\hbar,N:1}(t,x_1)dx_1\,,
\ea
$$
and, by the same token,
$$
\ba
\frac1{N^2}\int_{(\bR^3)^N}\sum_{1\le j\not=k\le N}w(x_i,x_j)\rho_{\hbar,N}(t,X_N)dX_N
\\
=\frac{N(N-1)}{N^2}\int_{(\bR^3)^N}w(x_1,x_2)\rho_{\hbar,N}(t,X_N)dX_N
\\
=\iint_{\bR^6}w(x_1,x_2)\tfrac{N-1}N\rho_{\hbar,N:2}(t,x_1,x_2)dx_1dx_2&\,.
\ea
$$
In other words,
$$
\ba
\int_{(\bR^3)^N}\!\!\iint_{\bR^6\setminus\bD}w(x,y)(\mu_{X_N}(dx)\!-\!\rho(t,dx))(\mu_{X_N}(dy)\!-\!\rho(t,dy))\rho_{\hbar,N}(t,X_N)dX_N
\\
=\int_{(\bR^3)^N}w(x_1,x_2)\big(\tfrac{N-1}N\rho_{\hbar,N:2}(t,x_1,x_2)+\rho(t,x_1)\rho(t,x_2)
\\
-2\rho_{\hbar,N:1}(t,x_1)\rho(t,x_2)\big)dx_1dx_2&\,.
\ea
$$
We conclude by applying this identity successively with $w(x,y)=V(x-y)$ and $w(x,y)=(u(t,x)-u(t,y))\cdot\grad V(x-y)$.
\end{proof}

\smallskip
At this point, we recall Serfaty's remarkable inequality, stated as Proposition 2.3 in \cite{DS}, in the case of the Coulomb potential in space dimension $3$ (i.e. for $s=1$ and $d=3$ in the notation of \cite{DS}).
This inequality is based on a very clever renormalization of the self-interaction of each particle, a precursor of which can be found in \cite{RougSerfaty}. However, the inequality below involves new important
ideas, perhaps the most important of which is a smearing of each point charge with a width that is not uniform for all charges: see formula (3.11) in \cite{DS}.

\bigskip
\noindent
\fbox{\sc Serfaty's inequality}

\smallskip
There exists $C>2$ such that, for all $\rho\in L^\infty(\bR^3)$, all $u\in W^{1,\infty}(\bR^3;\bR^3)$ and a.e. $X_N\in(\bR^3)^N$
\be\lb{SSIneq}
\ba
|F'_N(X_N,(\rho,u))|\le&C\|\grad u\|_{L^\infty(\bR^3)}F_N(X_N,\rho)
\\
&+C\left(1+\|\rho\|_{L^\infty(\bR^3)}\right)\left(1+\|u\|_{W^{1,\infty}(\bR^3)}\right)N^{5/3}\,.
\ea
\ee

\smallskip 
The functional $F(X_N,\rho(t,\cdot))$ is defined in terms of an integral over $\bR^6\setminus\bD$. Removing $\bD$ in the domain of integration is used in order to obtain a finite quantity; however the quantity
so defined may not be always nonnegative. This is fixed with the following inequality, taken from Corollary 3.4 \cite{DS}: there exists $C'>0$ such that
\be\lb{LBF}
F_N(X_N,\rho)\ge -C'(1+\|\rho\|_{L^\infty(\bR^3)})N^{4/3}\,.
\ee

\bigskip
Several remarks are in order at this point before going further in the proof of Theorem \ref{T-MFSCLim}. Most of these remarks are aimed at explaining the analogies and the differences between our approach
in the present work and the analysis in \cite{DS}.

\bigskip
\noindent
\textbf{Remarks.}

\smallskip
\noindent
(1) Serfaty's inequality \eqref{SSIneq} and \eqref{LBF} are the only ingredients from \cite{DS} used in the proof of Theorem \ref{T-MFSCLim}. This observation is obviously not aimed at diminishing the importance 
of \cite{DS} for our work, which is considerable, in the first place because \eqref{SSIneq} is the key which allows us to handle the Coulomb case. However, we think it worthwhile to point at some noticeable differences 
between the use of this inequality in the classical setting considered in \cite{DS} and in the quantum setting discussed here.

\smallskip
In \cite{DS}, the modulated potential energy $F_N(X_N,\rho(t,\cdot))$ is differentiated in time along a trajectory $t\mapsto X_N(t)$ of the gradient flow of the interaction energy
$$
\frac1N\sum_{1\le j<k\le N}V(x_j-x_k)\,.
$$
Similarly, in the appendix of \cite{DS}, the modulated total energy 
$$
\sum_{j=1}^N|\xi_j(t)-u(t,x_j(t))|^2+F(X_N(t),\rho(t,\cdot))
$$
is differentiated in time assuming that $t\mapsto (\xi_1(t),\ldots,\xi_N(t);X_N(t))$ is a phase-space trajectory of the $N$-particle Hamiltonian
$$
\sum_{j=1}^N|\xi_j|^2+\frac1N\sum_{1\le j<k\le N}V(x_j-x_k)\,.
$$
The inequality \eqref{SSIneq} is used to control the time-derivative $|\tfrac{d}{dt}F_N(X_N(t),\rho(t,\cdot))|$ in terms of the modulated potential energy $F_N(X_N(t),\rho(t,\cdot))$ itself, and therefore to bound
$F_N(X_N(t),\rho(t,\cdot))$ in terms of $F_N(X_N(0),\rho(0,\cdot))$ via Gronwall's lemma.

At variance with \cite{DS}, in the quantum setting considered here, the modulated potential energy $F_N(X_N,\rho(t,\cdot))$ is not used dynamically, since there is no notion of particle trajectory in quantum mechanics. 
Instead, the modulated potential energy $F_N(X_N,\rho(t,\cdot))$ enters the functional $\cF[R_{\hbar,N},\rho]$ as a test function, and \eqref{SSIneq} is used by duality. Moreover, the evolution of $\cF[R_{\hbar,N},\rho]$ 
is based on the $N$-particle quantum dynamics \eqref{NvN}, which differs from Newton's 2nd law written for each one of the $N$ particles. This is why the evolution of the functional $\cE[R_{\hbar,N},\rho,u]$, reported
in Proposition \ref{P-DtE}, must be established independently from the analogous computation in \cite{DS} --- notice in particular the $O(\hbar^2)$ term \eqref{h2term} which appears only in the quantum problem.

\smallskip
\noindent
(2) A simplifying feature in the definition of $\cE[R_{\hbar,N},\rho,u]$ is that its interaction part $\cF[R_{\hbar,N},\rho]$ involves only multiplication operators (i.e. potentials). This is precisely the reason why, even in 
the quantum setting, everything can be expressed in terms of the density \textit{function } $\rho_{\hbar,N}$ (instead of the density \textit{operator } $R_{\hbar,N}$) and of the empirical measure $\mu_{X_N}$ defined 
in Lemma \ref{L-FcF}. As a result, Serfaty's inequality \eqref{SSIneq} can be used in the quantum setting without the slightest modification.

\smallskip
\noindent
(3) If one wishes to follow exactly the computation in the appendix of \cite{DS}, one could probably use, instead of the functional $\cE[R_{\hbar,N},\rho,u]$, the formalism of ``quantum empirical measures'' recently 
defined in \cite{FGTPaulEmpir}. The ``quantum empirical measure'' is a time dependent continuous linear mapping 
$$
t\mapsto\cM_N(t)\in\cL(\cL(\fH),\cL_s(\fH_N))
$$ 
such that
$$
\Tr_{\fH_N}(R_{\hbar,N}(0)\cM_N(t)(A))=\Tr_{\fH}(R_{\hbar,N:1}(t)A)
$$
for each $A\in\cL(\fH)$. The quantity analogous to the functional $\tfrac1{N^2}H(Z_N,(\rho,u))$ considered in the appendix of \cite{DS} would be defined in terms of $\cM_N(t)$ as follows:
$$
''\,\,\cM_N(t)(|\hbar D_x-u(t,x)|^2)+\int_{\bR^3}\hat V(\om)|\cM_N(t)(E_\om)-\hat\rho(t,\om)I_{\fH_N})|^2\tfrac{d\om}{(2\pi)^d}\,\,,''
$$
where $E_\om\psi(x):=e^{i\om\cdot x}\psi(x)$ for all $\psi\in\fH$. This leads to several potentially unpleasant sources of technicalities which must be addressed. Firstly, $\cM_N(t)(A)$ is defined for all $A\in\cL(\fH)$, 
but the operator $|\hbar D_x-u(t,x)|^2$ is obviously not bounded. Therefore, even the definition of the first term in the functional above requires some additional care (for instance replacing $D_x$ with $D^\eps_x$ 
as done in Step 2 above). Another potential source of technical difficulties is found in the definition of the second term, where removing the diagonal from the domain of integration has to be translated somehow in 
terms of Fourier variables. Finally, the dynamical equation satisfied by $\cM_N(t)$ (equation (34) in \cite{FGTPaulEmpir}) has been established under the assumption that $\hat V\in L^1(\bR^3)$, a condition which
is obviously not verified by the Coulomb potential). Even after all these technical difficulties are handled satisfyingly, one should notice that, if $a\equiv a(x)$ is an element of $L^\infty(\bR^3)$, viewed as a bounded 
multiplication operator on $\fH$, then
$$
\ba
\Tr_{\fH_N}(R_{\hbar,N}(0)\cM_N(t)(a))&=\int_{\bR^3}a(x_1)\rho_{\hbar,N:1}(t,x_1)dx_1
\\
&=\int_{(\bR^3)^N}\la\mu_{X_N},a\ra\rho_{\hbar,N}(t,X_N)dX_N\,.
\ea
$$
In other words, since the interaction part of the modulated energy involves only multiplication operators, there is little more content in the quantum empirical measure $\cM_N(t)$  that in its classical analogue 
$\mu_{X_N}$. For all these reasons, we have chosen to avoid using this formalism here.

\smallskip
\noindent
(4) A last remark may be necessary for readers acquainted with mean-field limits in the classical setting, but unfamiliar with the formalism of BBGKY hierarchies. The formulation of the mean-field limit in terms 
of phase-space empirical measures (i.e. Klimontovich solutions) in the classical setting \cite{BraunHepp,Dobrushin} is in duality with the formulation of the same limit in terms of the BBGKY hierarchy: see for
instance Theorem 3.1 in \cite{FGMouhotRicci}. The connection between the quantum modulated energy $\cE[R_{\hbar,N},\rho,u]$ and the modulated potential energy $F(X_N,\rho(t,\cdot))$ from \cite{DS},
presented in Lemma \ref{L-FcF}, is explained by the quantum analogue of formula (32) in \cite{FGMouhotRicci} in the case $m=2$. In particular, this formula accounts for the combinatorial factor $\tfrac{N-1}N$
used in the definition of $\cE[R_{\hbar,N},\rho,u]$. While this term might seem strange at first sight because $\cF[R_{\hbar,N},\rho]$ does not vanish identically when $R_{\hbar,N}=(R_{\hbar})^{\otimes N}$,
in spite of the fact that such factorized states are supposedly the most favorable in the context of the mean-field limit, the $\tfrac{N-1}N$ factor definitely helps when computing the time-derivative of 
$\cE[R_{\hbar,N},\rho,u]$. More precisely, in the BBGKY formalism (see for instance section 1.10.1 in \cite{FGTwente}), the equation satisfied by $R_{\hbar,N:1}$ involves $R_{\hbar,N:2}$, the equation
satisfied by $R_{\hbar,N:2}$ involves $R_{\hbar,N:3}$ and so on. No marginal density operator $R_{\hbar,N:k}$ for $k=1,\ldots,N-1$ is expected to satisfy a closed equation, because of the interaction. It is 
a quite remarkable feature of the modulated energy functional $\cE[R_{\hbar,N},\rho,u]$, which involves the two first marginal density operators $R_{\hbar,N:1}$ and $R_{\hbar,N:2}$, that its time-derivative
is also expressed in terms of $R_{\hbar,N:1}$ and $R_{\hbar,N:2}$, and does not involves $R_{\hbar,N:3}$. In fact, using the specific combination 
$$
\tfrac{N-1}N\rho_{\hbar:2}(t,x_1,x_2)-2\rho_{\hbar:1}(t,x_1)\rho(t,x_2)+\rho(t,x_1)\rho(t,x_2)
$$
in the definition of $\cE[R_{\hbar,N},\rho,u]$ is precisely the reason why its evolution can be controlled by a Gronwall inequality, and not by a Cauchy-Kovalevska type argument, as is usual in most results 
based on BBGKY hierarchy techniques (see \cite{BEGMY} or section 1.11.2 in \cite{FGTwente} for a presentation of such arguments in the context described here, or \cite{Ukai} for the original presentation
of Cauchy-Kovalevska arguments on a different, albeit formally similar problem, namely the rigorous justification of the Boltzmann equation from the classical dynamics of $N$ identical hard spheres in the
so-called Boltzmann-Grad limit).

\bigskip
Returning to Proposition \ref{P-DtE}, we recast the variation in time of the modulated energy as follows, with the help of Lemma \ref{L-FcF}:
$$
\ba
\Tr_\fH(R_{\hbar,N:1}(t)^{1/2}|\hbar D_x-u(t,x)|^2R_{\hbar,N:1}(t)^{1/2})
\\
+\frac1{N^2}\int_{(\bR^3)^N}F_N(X_N,\rho(t,\cdot))\rho_{\hbar,N}(t,X_N)dX_N
\\
=\Tr_\fH(R_{\hbar,N:1}(0)^{1/2}|\hbar D_x-u(0,x)|^2R_{\hbar,N:1}(0)^{1/2})
\\
+\frac1{N^2}\int_{(\bR^3)^N}F_N(X_N,\rho(0,\cdot))\rho_{\hbar,N}(0,X_N)dX_N
\\
-2\int_0^t\Tr\left(R_{\hbar,N:1}(s)^{\frac12}(\hbar D_k-u_k)\Si_{jk}(\hbar D_j-u_j)R_{\hbar,N:1}(s)^{\frac12}\right)ds
\\
+\frac1{N^2}\int_0^t\int_{(\bR^3)^N}F'_N(X_N,(\rho,u)(s,\cdot))\rho_{\hbar,N}(s,X_N)dX_Nds
\\
+\tfrac12\hbar^2\int_0^t\int_{\bR^3}(\Dlt_x\Div_xu)(s,x)\rho_{\hbar,N:1}(s,x)dxds&\,.
\ea
$$
With \eqref{LBF}, since $\rho_{\hbar,N:1}(t,\cdot)$ is a probability density for all $t\in\bR$, so that its integral is $1$, we transform this identity into
$$
\ba
\Tr_\fH(R_{\hbar,N:1}(t)^{1/2}|\hbar D_x-u(t,x)|^2R_{\hbar,N:1}(t)^{1/2})
\\
+\frac1{N^2}\int_{(\bR^3)^N}\left(F_N(X_N,\rho(t,\cdot))+C'(1+\|\rho\|_{L^\infty([0,T]\times\bR^3)})N^{4/3}\right)\rho_{\hbar,N}(t,X_N)dX_N
\\
\le\Tr_\fH(R_{\hbar,N:1}(0)^{1/2}|\hbar D_x-u(0,x)|^2R_{\hbar,N:1}(0)^{1/2})
\\
+\frac1{N^2}\int_{(\bR^3)^N}\left(F_N(X_N,\rho(0,\cdot))+C'(1+\|\rho\|_{L^\infty([0,T]\times\bR^3)})N^{4/3}\right)\rho_{\hbar,N}(0,X_N)dX_N
\\
+2\int_0^t\left|\Tr\left(R_{\hbar,N:1}(s)^{\frac12}(\hbar D_k-u_k)\Si_{jk}(\hbar D_j-u_j)R_{\hbar,N:1}(s)^{\frac12}\right)\right|ds
\\
+\frac1{N^2}\int_{(\bR^3)^N}|F'_N(X_N,(\rho,u)(t,\cdot))|\rho_{\hbar,N}(t,X_N)dX_N
\\
+\tfrac12\hbar^2\int_0^t\int_{\bR^3}\left|(\Dlt_x\Div_xu)(s,x)\right|\rho_{\hbar,N:1}(s,x)dxds&\,.
\ea
$$
Using successively \eqref{Kin'<CKin}, Serfaty's inequality above, \eqref{LBF} and \eqref{h2term}, we find that
$$
\ba
\Tr_\fH(R_{\hbar,N:1}(t)^{1/2}|\hbar D_x-u(t,x)|^2R_{\hbar,N:1}(t)^{1/2})
\\
+\int_{(\bR^3)^N}\left(\frac{F_N(X_N,\rho(t,\cdot))}{N^2}+\frac{C'(1+\|\rho\|_{L^\infty([0,T]\times\bR^3)})}{N^{2/3}}\right)\rho_{\hbar,N}(t,X_N)dX_N
\\
\le\Tr_\fH(R_{\hbar,N:1}(0)^{1/2}|\hbar D_x-u(0,x)|^2R_{\hbar,N:1}(0)^{1/2})
\\
+\int_{(\bR^3)^N}\left(\frac{F_N(X_N,\rho(0,\cdot))}{N^2}+\frac{C'(1+\|\rho\|_{L^\infty([0,T]\times\bR^3)})}{N^{2/3}}\right)\rho_{\hbar,N}(0,X_N)dX_N
\\
+2\sup_{1\le j,k\le 3}\|\Si_{jk}\|_{L^\infty([0,T]\times\bR^3)}\int_0^t\Tr\left(R_{\hbar,N:1}(s)^{\frac12}|\hbar D_k-u_k|^2R_{\hbar,N:1}(s)^{\frac12}\right)ds
\\
+C\|\grad u(t,\cdot)\|_{L^\infty(\bR^3)}\int_0^t\int_{(\bR^3)^N}\frac{F_N(X_N,(\rho,u)(t,\cdot))}{N^2}\rho_{\hbar,N}(t,X_N)dX_N
\\
+\frac{Ct}{N^{1/3}}\left(1+\|\rho\|_{L^\infty([0,T]\times\bR^3)}\right)\left(1+\|u\|_{W^{1,\infty}([0,T]\times\bR^3)}\right)
\\
+\tfrac12\hbar^2t\|\Dlt_x\Div_xu\|_{L^\infty([0,T]\times\bR^3)}&\,.
\ea
$$
Applying Gronwall's lemma leads to the following intermediate statement.

\begin{Prop}\lb{P-PropaEstim}
Let $R_{\hbar,N}^{in}\in\cD_s(\fH_N)$ satisfy 
$$
\Tr_{\fH_N}\left((R_{\hbar,N}^{in})^{1/2}\left(I+\sum_{j=1}^N-\tfrac12\hbar^2\Dlt_j\right)^2(R_{\hbar,N}^{in})^{1/2}\right)<\infty\,,
$$
and let $R_{\hbar,N}(t):=e^{-it\scrH_N/\hbar}R_N^{in}e^{+it\scrH_N/\hbar}$. Then, for each $t\in[0,T]$, each $\hbar>0$ and each $N>1$, one has
$$
\ba
\Tr_\fH(R_{\hbar,N:1}(t)^{1/2}|\hbar D_x-u(t,x)|^2R_{\hbar,N:1}(t)^{1/2})
\\
+\int_{(\bR^3)^N}\left(\frac{F_N(X_N,\rho(t,\cdot))}{N^2}+\frac{C'(1+\|\rho\|_{L^\infty([0,T]\times\bR^3)})}{N^{2/3}}\right)\rho_{\hbar,N}(t,X_N)dX_N
\\
\le\exp(CT\|\grad u\|_{L^\infty([0,T]\times\bR^3)})\Bigg(\Tr_\fH((R^{in}_{\hbar,N:1})^{1/2}|\hbar D_x-u^{in}(x)|^2(R^{in}_{\hbar,N:1})^{1/2})
\\
+\int_{(\bR^3)^N}\left(\frac{F_N(X_N,\rho^{in})}{N^2}+\frac{C'(1+\|\rho\|_{L^\infty([0,T]\times\bR^3)})}{N^{2/3}}\right)\rho^{in}_{\hbar,N}(X_N)dX_N\Bigg)
\\
+\frac{CT\exp(CT\|\grad u\|_{L^\infty([0,T]\times\bR^3)})}{N^{1/3}}\left(1+\|\rho\|_{L^\infty([0,T]\times\bR^3)}\right)\left(1+\|u\|_{W^{1,\infty}([0,T]\times\bR^3)}\right)
\\
+\tfrac12\hbar^2T\exp(CT\|\grad u\|_{L^\infty([0,T]\times\bR^3)})\|\Dlt_x\Div_xu\|_{L^\infty([0,T]\times\bR^3)}&\,.
\ea
$$
\end{Prop}


\section{Proof of Theorem \ref{T-MFSCLim}}\lb{S-ProofT}


The propagation estimate obtained in Proposition \ref{P-PropaEstim} is by far the most important part of the proof of Theorem \ref{T-MFSCLim}. How to conclude from this estimate differs noticeably from 
the argument in \cite{DS}. Indeed, this argument is based on Proposition 3.5 in \cite{DS}, by which a weak seminorm of the difference between $\rho(t,\cdot)$ and the empirical measure of the $N$-tuple 
of particle positions is bounded in terms of $F_N/N^2$ modulo a $o(1)$ term. More precisely
$$
\left|\La\tfrac1N\sum_{j=1}^N\de_{x_j(t)}\!-\!\rho(t,\cdot),\phi\Ra\right|\!\!\le\!C\|\phi\|_{C^{0,\a}\cap\dot H^1(\bR^3)}\bigg(\frac{F_N(X_N(t),\rho(t))}{N^2}\!+\!\frac{1\!\!+\!\!\|\rho(t,\cdot)\|_{L^\infty(\bR^3)}}{N^{2/3}}\bigg).
$$
This argument cannot be adapted to the quantum setting considered here (in the first place, because there is no analogue of the particle trajectories $t\mapsto X_N(t)$). Our approach differs again from that of \cite{DS}
from now on.

\subsection{Step 1: vanishing of the modulated energy}


Assumption \eqref{EPrho0} implies in particular that 
$$
\iint_{\bR^6}\frac{\rho^{in}_{\hbar}(x)\rho^{in}_{\hbar}(y)}{|x-y|}dxdy<\infty\,.
$$
Moreover, since $\rho^{in}\in L^1\cap L^\infty(\bR^3)$, the function $V\star\rho^{in}$ belongs to $C_b(\bR^3)$ (the set of bounded continuous functions on $\bR^3$), so that
$$
\int_{\bR^3}\rho^{in}_{\hbar}(x)|V\star\rho^{in}(x)|dx\le\|V\star\rho^{in}\|_{L^\infty(\bR^3)}\int_{\bR^3}\rho^{in}_{\hbar}(x)dx=\|V\star\rho^{in}\|_{L^\infty(\bR^3)}\,.
$$
Hence \eqref{EPrho0} implies that
$$
\iint_{\bR^6}\frac{\rho^{in}_{\hbar}(x)\rho^{in}_{\hbar}(y)}{|x-y|}dxdy\le 3\|V\star\rho^{in}\|_{L^\infty(\bR^3)}+o(1)
$$
as $\hbar\to 0$ --- in particular
$$
\sup_{0<\hbar<1}\iint_{\bR^6}\frac{\rho^{in}_{\hbar}(x)\rho^{in}_{\hbar}(y)}{|x-y|}dxdy<\infty\,.
$$
Hence
$$
\ba
\cF[R_{\hbar,N},\rho](0)
\\
=\iint_{\bR^6}V(x_1-x_2)\left(\tfrac{N-1}N\rho^{in}_{\hbar}(x_1)\rho^{in}_{\hbar}(x_2)+\rho^{in}(x_1)\rho^{in}(x_2)-2\rho^{in}_{\hbar}(x_1)\rho^{in}(x_2)\right)dx_1dx_2
\\
=\iint_{\bR^6}V(x_1-x_2)(\rho^{in}_{\hbar}(x_1)-\rho^{in}(x_1))(\rho^{in}_{\hbar}(x_2)-\rho^{in}(x_2))dx_1dx_2
\\
-\frac1N\iint_{\bR^6}V(x_1-x_2)\rho^{in}_{\hbar}(x_1)\rho^{in}_{\hbar}(x_2)dx_1dx_2=o(1)+O(1/N)
\ea
$$
as $\tfrac1N+\hbar\to 0$. With \eqref{Monokin0}, this implies that
$$
\cE[R_{\hbar,N},\rho,u](0)\to 0\qquad\text{ as }\tfrac1N+\hbar\to 0\,.
$$
By Proposition \ref{P-PropaEstim} and Lemma \ref{L-FcF}, we conclude that, for each $t\in[0,T]$, one has
\be\lb{Eto0}
\cE[R_{\hbar,N},\rho,u](t)\to 0\qquad\text{ as }\tfrac1N+\hbar\to 0\,.
\ee

\subsection{Step 2: Proof of \eqref{Cvrhot}}


Equivalently
$$
\cE[R_{\hbar,N},\rho,u](t)+\frac{C'(1+\|\rho\|_{L^\infty([0,T]\times\bR^3)})}{N^{2/3}}\to 0\qquad\text{ as }\tfrac1N+\hbar\to 0\,,
$$
for all $t\in[0,T]$, which implies in particular that both
\be\lb{VanishF}
\cF[R_{\hbar,N},\rho,](t)+\frac{C'(1+\|\rho\|_{L^\infty([0,T]\times\bR^3)})}{N^{2/3}}\to 0\qquad\text{ as }\tfrac1N+\hbar\to 0\,,
\ee
and
\be\lb{VanishMonokin}
\Tr_\fH(R_{\hbar,N:1}(t)^{1/2}|\hbar D_x-u(t,x)|^2R_{\hbar,N:1}(t)^{1/2})\to 0\qquad\text{ as }\tfrac1N+\hbar\to 0\,,
\ee
in the limit as $\tfrac1N+\hbar\to 0$ for all $t\in[0,T]$. In the present section, we analyze the consequences of \eqref{VanishF}, and postpone the discussion of \eqref{VanishMonokin} until the next section.

\bigskip
It will be convenient to use the following representation of the Coulomb potential.

\bigskip
\noindent
\fbox{\sc Smooth variant of the Fefferman-de la Llave formula} 

For each $x\not=y\in\bR^3$, one has
\be\lb{smFdlL}
V(x-y)=\frac1{4\pi|x-y|}=\int_0^\infty\int_{\bR^3}G_{r}(x-z)G_r(y-z)dzdr\,,
\ee
with the notation
$$
G_r(x):=(2\pi r)^{-3/2}\exp(-|x|^2/2r)\,.
$$

\bigskip
This is formula (3.3) of \cite{PRSS}. (Here is a quick ``formal'' argument to recover this formula: first, one has
$$
\int_0^\infty e^{r\Dlt}dr=(-\Dlt)^{-1}\,.
$$
Next we recall that
$$
\text{integral kernel of }(-\Dlt)^{-1}=\frac1{4\pi|x-y|}\,,\quad\text{ integral kernel of }e^{-r\Dlt/2}=G_r(x-y)\,.
$$
Hence
$$
\int_{\bR^3}G_{r}(x-z)G_r(y-z)dz=\text{ integral kernel of }e^{r\Dlt/2}e^{r\Dlt/2}=e^{r\Dlt}\,,
$$
and the the desired formula follows from the first identity above.) For the sake of being complete, we also recall the original Fefferman-de la Llave formula
\be\lb{FdlL}
V(x-y)=\tfrac1{4\pi^2}\int_0^\infty\frac1{r^5}\left(\int_{\bR^3}\indc_{|x-z|<r}\indc_{|y-z|<r}dz\right)dr\,,\quad\text{ for all }x\not=y\in\bR^3\,.
\ee
However, in the present paper, we shall use the smooth variant \eqref{smFdlL} of the Fefferman-de la Llave formula rather than the original formula \eqref{FdlL}.

\bigskip
For $\eta>0$, set
$$
V_\eta(x-y):=\int_\eta^\infty\int_{\bR^3}G_{r}(x-z)G_r(y-z)dzdr\,,
$$
which is the integral kernel of $e^{\eta\Dlt}(-\Dlt)^{-1}$, i.e.
$$
V_\eta(X)=\int_{\bR^3}G_\eta(X-Y)\frac{dY}{4\pi|Y|}\,.
$$
Since $\rho_{\hbar,N:2}(t,x_1,x_2)\ge 0$ for a.e. $(x_1,x_2)\in\bR^6$ and $\rho(t,x)\ge 0$ for all $x\in\bR^3$, one has
$$
\ba
\cF[R_{\hbar,N},\rho](t)\ge\iint_{\bR^6}V_\eta(x_1-x_2)
\\
\times\big(\tfrac{N-1}N\rho_{\hbar,N:2}(t,x_1,x_2)\!-\!2\rho_{\hbar,N:1}(t,x_1)\rho(t,x_2)\!+\!\!\rho(t,x_1)\rho(t,x_2)\big)dx_1dx_2
\\
-2\iint_{\bR^6}(V-V_\eta)(x_1-x_2)\rho_{\hbar,N:1}(t,x_1)\rho(t,x_2)dx_1dx_2
\\
=\cF_{1,\eta}(t)+\cF_{2,\eta}(t)&\,.
\ea
$$

Observe that
$$
\ba
\iint_{\bR^6}(V-V_\eta)(x_1-x_2)\rho_{\hbar,N:1}(t,x_1)\rho(t,x_2)dx_1dx_2
\\
=\int_{\bR^3}\rho_{\hbar,N:1}(t,x_1)\left(\int_0^\eta\int_{\bR^3}G_{2r}(x_1-x_2)\rho(t,x_2)dx_2dr\right)dx_1
\\
=\int_{\bR^3}\rho_{\hbar,N:1}(t,x_1)\left(\int_0^\eta e^{r\Dlt}\rho(t,x_1)dr\right)dx_1&\,,
\ea
$$
so that
$$
\ba
\left|\iint_{\bR^6}(V-V_\eta)(x_1-x_2)\rho_{\hbar,N:1}(t,x_1)\rho(t,x_2)dx_1dx_2\right|
\\
\le\|\rho_{\hbar,N:1}(t,\cdot)\|_{L^1(\bR^3)}\left\|\int_0^\eta e^{r\Dlt}\rho(t,x_1)dr\right\|_{L^\infty(\bR^3)}
\\
\le\|\rho_{\hbar,N:1}(t,\cdot)\|_{L^1(\bR^3)}\int_0^\eta\left\|e^{r\Dlt}\rho(t,x_1)\right\|_{L^\infty(\bR^3)}dr
\\
\le\eta\|\rho_{\hbar,N:1}(t,\cdot)\|_{L^1(\bR^3)}\|\rho(t,.)\|_{L^\infty(\bR^3)}
\\
=\eta\|\rho\|_{L^\infty([0,T]\times\bR^3)}&\,,
\ea
$$
since the heat equation satisfies the Maximum Principle. Thus
$$
|\cF_{2,\eta}(t)|\le 2\eta\|\rho\|_{L^\infty([0,T]\times\bR^3)}\,.
$$

Now for $\cF_{1,\eta}(t)$. One has
$$
\ba
\int_{(\bR^3)^N}\rho_{\hbar,N}(t,X_N)\iint_{\bR^6}V_\eta(q-q')\left(\frac1N\sum_{j=1}^N\de_{x_j}-\rho(t,\cdot)\right)^{\otimes 2}\!\!\!\!\!(dqdq')dX_N
\\
=\iint_{\bR^6}\tfrac{N-1}N\rho_{\hbar,N:2}(t,x_1,x_2)V_\eta(x_1-x_2)dx_1dx_2+\tfrac1N V_\eta(0)
\\
-2\iint_{\bR^6}\rho_{\hbar,N:1}(t,x_1)\rho(t,x_2)V_\eta(x_1-x_2)dx_1dx_2
\\
+\iint_{\bR^6}\rho(t,x_1)\rho(t,x_2)V_\eta(x_1-x_2)dx_1dx_2
\\
=\cF_{1,\eta}(t)+\tfrac1N V_\eta(0)&\,.
\ea
$$
On the other hand, using the smooth Fefferman-de la Llave representation \eqref{smFdlL} truncated for $r>\eta$ shows that
$$
\ba
\iint_{\bR^6}V_\eta(q-q')\left(\frac1N\sum_{j=1}^N\de_{x_j}-\rho(t,\cdot)\right)^{\otimes 2}\!\!\!\!\!(dqdq')
\\
=\int_\eta^\infty\int_{\bR^3}\left(\int_{\bR^3}G_r(q-z)\left(\frac1N\sum_{j=1}^N\de_{x_j}-\rho(t,\cdot)\right)(dq)\right)^2dzdr&\,.
\ea
$$
Now, for each $z\in\bR^3$ and $r>\eta$, one has
$$
\ba
\int_{(\bR^3)^N}\rho_{\hbar,N}(t,X_N)\left(\int_{\bR^3}G_r(q-z)\left(\frac1N\sum_{j=1}^N\de_{x_j}-\rho(t,\cdot)\right)(dq)\right)^2dX_N
\\
\ge\left(\int_{(\bR^3)^N}\rho_{\hbar,N}(t,X_N)\left(\int_{\bR^3}G_r(q-z)\left(\frac1N\sum_{j=1}^N\de_{x_j}-\rho(t,\cdot)\right)(dq)\right)dX_N\right)^2
\\
=\left(\int_{\bR^3}(\rho_{\hbar,N:1}(t,x_1)-\rho(t,x_1))G_r(x_1-z)dx_1\right)^2&\,.
\ea
$$
Thus
$$
\int_\eta^\infty\int_{\bR^3}\left(\int_{\bR^3}(\rho_{\hbar,N:1}(t,x_1)-\rho(t,x_1))G_r(x_1-z)dx_1\right)^2dzdr\le\cF_{1,\eta}(t)+\tfrac1N V_\eta(0)\,,
$$
or equivalently
$$
\ba
\int_\eta^\infty&\left\|e^{r\Dlt/2}(\rho_{\hbar,N:1}(t,\cdot)-\rho(t,\cdot))\right\|^2_{L^2(\bR^3)}dr
\\
&\le\cF[R_{\hbar,N},\rho](t)+\eta\|\rho\|_{L^\infty([0,T]\times\bR^3)}+\tfrac1N V_\eta(0)\,.
\ea
$$
Since
$$
V_\eta(0)=\int_{\bR^3}G_\eta(Y)\frac{dY}{4\pi|Y|}=\frac1{\sqrt{\eta}}\int_{\bR^3}G_1(y)\frac{dy}{4\pi|y|}=\frac{v_0}{\sqrt{\eta}}\,,
$$
the inequality above shows that, for each $\eps>0$ and each $N>v_0/\eps^{3/2}\|\rho\|_{L^\infty([0,T]\times\bR^3)}$, one has
$$
\ba
\int_\eps^\infty\left\|e^{r\Dlt/2}(\rho_{\hbar,N:1}(t,\cdot)-\rho(t,\cdot))\right\|^2_{L^2(\bR^3)}dr
\\
\le \cF[R_{\hbar,N},\rho](t)+\inf_{0<\eta<\eps}\left(\eta\|\rho\|_{L^\infty([0,T]\times\bR^3)}+\tfrac1N V_\eta(0)\right)
\\
\le \cF[R_{\hbar,N},\rho](t)+2\frac{v_0^{2/3}\|\rho\|^{1/3}_{L^\infty([0,T]\times\bR^3)}}{N^{2/3}}&\,.
\ea
$$
With \eqref{VanishF}, we conclude that
\be\lb{CvRhoHeat}
\int_\eps^\infty\left\|e^{r\Dlt/2}(\rho_{\hbar,N:1}(t,\cdot)-\rho(t,\cdot))\right\|^2_{L^2(\bR^3)}dr\to 0
\ee
for each $\eps>0$ and each $t\in[0,T]$ as $\tfrac1N+\hbar\to 0$.

On the other hand, we know that, for each $t\in[0,T]$, the family $\rho_{\hbar,N:1}(t,\cdot)$ satisfies
$$
\rho_{\hbar,N:1}(t,x)\ge 0\text{ for a.e. }x\in\bR^3\,,\quad\int_{\bR^3}\rho_{\hbar,N:1}(t,x)dx=1\,.
$$
By the Banach-Alaoglu theorem (see Theorem 3.1.6 of \cite{Brezis}), this family is relatively compact for the weak topology of bounded Borel measures on $\bR^3$. If 
$$
\rho_{\hbar_n,N_n:1}(t,\cdot)\to\bar\rho\equiv\bar\rho(dx)\quad\text{ as }\frac1{N_n}+\hbar_n\to 0
$$
in the weak topology of bounded Borel measures on $\bR^3$, we deduce from \eqref{CvRhoHeat} that
$$
e^{r\Dlt/2}(\bar\rho-\rho(t,\cdot))=0\quad\text{ for all }r>\eps\,.
$$
Hence $\bar\rho=\rho(t,\cdot)$, so that, by compactness and uniqueness of the limit,
$$
\rho_{\hbar,N:1}(t,\cdot)\to\rho(t,\cdot)\text{ weakly as }\frac1N+\hbar\to 0
$$
for all $t\in[0,T]$. Because of \eqref{NormaRhot} and Theorem 6.8 in \cite{Malliavin}, we conclude that the convergence above holds in the narrow topology, which proves \eqref{Cvrhot}.

\subsection{Step 3: Proof of \eqref{CvJt}}


The definition \eqref{DefCurrent} of the current 
$$
J_{\hbar,N:1}=(J^1_{\hbar,N:1},J^2_{\hbar,N:1},J^3_{\hbar,N:1})
$$ 
implies that, for each test vector field $b\equiv b(x)=(b_1(x),b_2(x),b_3(x))$ in $C_b(\bR^3;\bR^3)$ and $m=1,2,3$, one has
$$
\Tr\left(b_m(\tfrac12R_{\hbar,N:1}(t)\vee(\hbar D_m-u_m(t,\cdot))\right)=\la J^m_{\hbar,N:1}(t)-\rho_{\hbar,N:1}(t,\cdot)u_m(t,\cdot),b_m\ra
$$
for all $t\in[0,T]$. Therefore
$$
\ba
\sum_{m=1}^3&\left|\la J_{\hbar,N:1}^m(t)-\rho_{\hbar,N:1}(t,\cdot)u_m(t,\cdot),b_m\ra\right|
\\
&\le\!\sum_{m=1}^3|\Tr(b_m(\tfrac12R_{\hbar,N:1}(t)\vee(\hbar D_m\!-\!u_m(t,\cdot))))|
\\
&\le\sum_{m=1}^3\|R_{\hbar,N:1}(t)^{1/2}(\hbar D_m-u_m)\|_2\|b_mR_{\hbar,N:1}(t)^{1/2}\|_2
\\
&\le\sum_{m=1}^3\|R_{\hbar,N:1}(t)^{1/2}(\hbar D_m-u_m)\|_2\|b_m\|_{L^\infty(\bR^3)}\to 0
\ea
$$
for all $t\in[0,T]$ as $\tfrac1N+\hbar\to 0$ because of \eqref{VanishMonokin}. Thus, for each $t\in[0,T]$ and each $m=1,2,3$,
$$
J_{\hbar,N:1}^m(t)-\rho_{\hbar,N:1}(t,\cdot)u_m(t,\cdot)\to 0
$$
for the narrow topology of (bounded) signed Borel measures on $\bR^3$. With the convergence \eqref{Cvrhot} already established in the previous step, and since $u_m(t,\cdot)\in C_b(\bR^3)$ for each 
$t\in[0,T]$ and each $m=1,2,3$, this implies that \eqref{CvJt} holds.

\subsection{Step 4: Proof of \eqref{Wignert}}


The convergence \eqref{Monokint} is already established --- see \eqref{VanishMonokin} at the beginning of Step 2 above.

Because of \eqref{VanishMonokin}, one has
\be\lb{TightRhN1}
\sup_{N\ge 2\,,\,\,\hbar\in(0,1)}\Tr(R_{\hbar,N;1}(t)^{1/2}|\hbar D_x|^2R_{\hbar,N;1}(t)^{1/2})<\infty\,.
\ee
Let $R_{\hbar_n,N_n;1}(t)$ be a subsequence of $R_{\hbar,N:1}(t)$ such that
$$
W_{\hbar_n}[R_{\hbar_n,N_n:1}(t)]\to W(t)
$$
in $\cS'(\bR^3\times\bR^3)$ as $\tfrac1N_n+\hbar_n\to 0$. By Theorem III.2 and the bound (55) in \cite{LionsPaul}, 
\be\lb{intWdxi}
\rho(t,\cdot)=\int_{\bR^3}W(t)d\xi\,.
\ee
If $\nu:=\nu(dxd\xi)$ is a bounded measure on $\bR^d_x\times\bR^d_\xi$, we denote by
$$
\int_{\bR^d}\nu d\xi
$$
the push-forward of $\nu$ under the projection $\bR^d_x\times\bR^d_\xi\ni(x,\xi)\mapsto x\in\bR^d$. In other words, 
$$
\La\int_{\bR^d}\nu d\xi,\phi\Ra:=\iint_{\bR^d\times\bR^d}\phi(x)\nu(dxd\xi)
$$
for each test function $\phi\in C_b(\bR^d)$.) In particular
\be\lb{Wu2}
\int_{\bR^3}|u(t,x)|^2W(t,dxd\xi)=\int_{\bR^3}\rho(t,x)|u(t,x)|^2dx\,.
\ee

Let the Husimi transform of $R_{\hbar,N:1}(t)$ be defined by
$$
\tilde W_{\hbar}[R_{\hbar,N:1}(t)]:=e^{\hbar\Dlt_{x,\xi}/4}W_{\hbar}[R_{\hbar,N:1}(t)]\,.
$$
By formula (25) in \cite{LionsPaul}, one has $\tilde W_{\hbar}[R_{\hbar,N:1}(t)]\ge 0$, while
$$
\iint_{\bR^6}|\xi|^2\tilde W_{\hbar}[R_{\hbar,N:1}(t)](x,\xi)dxd\xi=\Tr(R_{\hbar,N:1}(t)^{1/2}|\hbar D_x|^2R_{\hbar,N:1}(t)^{1/2})+3\hbar
$$
according to the formula following equation (30) in \cite{LionsPaul} (with the only difference that section II in \cite{LionsPaul} assumes that $\hbar=1$). Specializing this to the subsequence $R_{\hbar_n,N_n;1}(t)$
and passing to the limit as $\tfrac1{N_n}+\hbar_n\to 0$, we conclude from Fatou's lemma that
\be\lb{Wxi2}
\ba
\iint_{\bR^6}|\xi|^2W(t,dxd\xi)\le&\varliminf_{1/N_n+\hbar_n\to 0}\iint_{\bR^6}|\xi|^2\tilde W_{\hbar_n}[R_{\hbar_n,N_n:1}(t)](x,\xi)dxd\xi
\\
=&\varliminf_{1/N_n+\hbar_n\to 0}\Tr(R_{\hbar_n,N_n:1}(t)^{1/2}|\hbar_nD_x|^2R_{\hbar_n,N_n:1}(t)^{1/2})\,.
\ea
\ee

Let $\chi\in C^\infty(\bR^3)$ satisfy $\indc_{B(0,1)}(\xi)\le\chi(\xi)\le\indc_{B(0,2)}(\xi)$ for all $\xi\in\bR^3$, and let $u^\eta\equiv u_\eta(x)\in\bR^3$ such that $u^\eta_j\in\cS(\bR^3)$ for all $j=1,2,3$  and $\eta>0$,
and
$$
\|u^\eta_j-u_j(t,\cdot)\|_{W^{1,\infty}(\bR^3)}\to 0\quad\text{ for }j=1,2,3\text{ as }\eta\to 0^+\,.
$$
Set 
$$
f_\eta(x,\xi):=\chi(\eta\xi)u^\eta(x)\cdot\xi\,,\qquad x,\xi\in\bR^3\,,\,\,\eta>0\,.
$$
Since $u^\eta_j\in\cS(\bR^3)$ for $j=1,2,3$, the function $f_\eta\in\cS(\bR^3\times\bR^3)$ for each $\eta>0$. Since $f_\eta(x,\hbar D_x)=u^\eta_j(x)\chi(\eta\hbar D_x)\hbar D_j$, applying Remark III.10 in 
\cite{LionsPaul} shows that
$$
\iint_{\bR^6}u^\eta_j(x)\xi_j\chi(\eta\hbar_n\xi)W(t,dxd\xi)dxd\xi=\lim_{n\to\infty}\Tr(u_j(x)\chi(\eta\hbar D_x)\hbar D_jR_{\hbar_n,N_n:1}(t))
$$
for each $\eta>0$. Let $(\Psi^n_m)_{m\ge 1}$ be a complete orthonormal system of eigenfunctions of $R_{\hbar_n,N_n:1}(t)$ in $L^2(\bR^3)$, and let $\L^n_m$ be the associated sequence of eigenvalues, i.e.
$R_{\hbar_n,N_n:1}(t)\Psi^n_m=\L^n_m\Psi^n_m$, ordered so that $\L^n_1\ge\L^n_2\ge\ldots\ge\L^n_m\ge\ldots\ge 0$. Using the Plancherel theorem, we see that the condition \eqref{TightRhN1} is expressed as 
$$
\sup_{n\ge 1}\tfrac1{(2\pi)^3}\sum_{m\ge 1}\L^n_m\|\hbar_n|\xi|\widehat{\Psi^n_m}\|^2_{L^2(\bR^3)}=\sup_{n\ge 1}\sum_{m\ge 1}\L^n_m\|\hbar_n|D_x|\Psi^n_m\|^2_{L^2(\bR^3)}=C<\infty\,,
$$
while
$$
\ba
{}&\Tr(u_j(t,\cdot)\hbar D_jR_{\hbar_n,N_n:1}(t))=\sum_{m\ge 1}\L^n_m\la\hbar D_j(u_j\Psi^n_m)|\Psi^n_m\ra\,,
\\
&\Tr(u^\eta_j(x)\chi(\eta\hbar D_x)\hbar D_jR_{\hbar_n,N_n:1}(t))=\sum_{m\ge 1}\L^n_k\la\hbar D_j(u^\eta_j\Psi^n_m)|\chi(\eta\hbar D_x)\Psi^n_m\ra\,.
\ea
$$
Observe that, assuming $0\le\hbar_n\le 1$,
$$
\ba
\left|\sum_{m\ge 1}\L^n_m\la\hbar_nD_j(u_j\Psi^n_m)|(1-\chi(\eta\hbar_nD_x))\Psi^n_m\ra\right|^2
\\
\le\sum_{m\ge 1}\L^n_m\|\hbar_nD_j(u_j\Psi^n_m)\|^2_{L^2(\bR^3)}\sum_{m\ge 1}\L^n_m\|(1-\chi(\eta\hbar_nD_x))\Psi^n_m\|^2_{L^2(\bR^3)}
\\
\le2\sum_{m\ge 1}\L^n_m\left(\|u\|^2_{L^\infty(\bR^3)}\|\hbar_n|D_j|\Psi^n_m\|^2_{L^2(\bR^3)}+\hbar_n^2\|\Div_xu\|^2_{L^\infty(\bR^3)}\|\Psi^n_m\|^2_{L^2(\bR^3)}\right)
\\
\times\tfrac1{(2\pi)^3}\sum_{m\ge 1}\L^n_m\|\indc_{\eta\hbar_n|\xi|\ge 1}\widehat{\Psi^n_m}\|^2_{L^2(\bR^3)}
\\
\le 2(\|u\|^2_{L^\infty(\bR^3)}\!+\!\hbar_n^2\|\Div_xu\|^2_{L^\infty(\bR^3)})\Tr(R_{\hbar,N:1}(t)^{1/2}(1\!+\!|\hbar_nD_x|^2)R_{\hbar,N:1}(t)^{1/2})
\\
\times\tfrac1{(2\pi)^3}\eta^2\sum_{m\ge 1}\L^n_m\|\hbar_n|\xi|\widehat{\Psi^n_m}\|^2_{L^2(\bR^3)}
\\
\le 2\eta^2\|u\|^2_{W^{1,\infty}(\bR^3)}\Tr(R_{\hbar,N:1}(t)^{1/2}(1+|\hbar_nD_x|^2)R_{\hbar,N:1}(t)^{1/2})^2&\,.
\ea
$$
On the other hand
$$
\ba
\left|\sum_{m\ge 1}\L^n_m\la\hbar_nD_j((u_j-u^\eta_j)\Psi^n_m)|\chi(\eta\hbar_nD_x)\Psi^n_m\ra\right|^2
\\
=\left|\sum_{m\ge 1}\L^n_m\la(u_j(t,\cdot)-u^\eta_j)\Psi^n_m|\chi(\eta\hbar_nD_x)\hbar_nD_j\Psi^n_m\ra\right|^2
\\
\le\sum_{j=1}^3\sum_{m\ge 1}\L^n_m\|(u_j-u^\eta_j)\Psi^n_m\|^2_{L^2(\bR^3)}\sum_{m\ge 1}\L^n_m\|\chi(\eta\hbar_nD_x)\hbar_nD_j\Psi^n_m\|^2_{L^2(\bR^3)}
\\
\le\sum_{j=1}^3\sum_{m\ge 1}\L^n_m\|(u_j-u^\eta_j)\Psi^n_m\|^2_{L^2(\bR^3)}\sum_{m\ge 1}\L^n_m\|\hbar_nD_j\Psi^n_m\|^2_{L^2(\bR^3)}
\\
\le 3C\|u(t,\cdot)-u^\eta\|_{L^\infty(\bR^3)}\sum_{m\ge 1}\L^n_m\|\Psi^n_m\|^2_{L^2(\bR^3)}=3C\|u(t,\cdot)-u^\eta\|_{L^\infty(\bR^3)}&\,.
\ea
$$

Hence
$$
\ba
\Tr&((u_j(t,\cdot)-u^\eta_j\chi(\eta\hbar D_x))\hbar D_jR_{\hbar_n,N_n:1}(t))
\\
=&\sum_{m\ge 1}\L^n_m\la\hbar_nD_j(u_j(t,\cdot)\Psi^n_m)|(1-\chi(\eta\hbar_nD_x))\Psi^n_m\ra
\\
&+\sum_{m\ge 1}\L^n_m\la(u_j(t,\cdot)-u^\eta_j)\Psi^n_m|\chi(\eta\hbar_nD_x)\hbar_nD_j\Psi^n_m\ra\to 0
\ea
$$
as $\eta\to 0$ uniformly in $n$, and therefore
\be\lb{Wuxi}
\iint_{\bR^6}u_j(t,x)\xi_jW(t,dxd\xi)dxd\xi=\lim_{n\to\infty}\Tr(u_j(t,\cdot)\hbar D_jR_{\hbar_n,N_n:1}(t))\,.
\ee

With \eqref{Wu2} and \eqref{Wxi2}, we conclude from \eqref{Wuxi} and \eqref{Monokin0} that
$$
\ba
\iint_{\bR^6}|\xi-u(t,x)|^2W(t,dxd\xi)
\\
\le\varliminf_{1/N_n+\hbar_n\to 0}\Tr(R_{\hbar_n,N_n:1}(t)^{1/2}|\hbar_nD_x-u(t,\cdot)|^2R_{\hbar_n,N_n:1}(t)^{1/2})=0&\,.
\ea
$$
Since $W(t)\ge 0$ by Theorem III.2 in \cite{LionsPaul}, we conclude that
$$
|\xi-u(t,x)|^2W(t,dxd\xi)=0\,.
$$
With \eqref{intWdxi}, this implies that $W(t,dxd\xi)=\rho(t,x)\de(\xi-u(t,x))$. Since all the limit points of $W_{\hbar}[R_{\hbar,N:1}(t)]$ are of this form, and $W_{\hbar}[R_{\hbar,N:1}(t)]$ is bounded in the topological 
dual $\cA'$ of the Lions-Paul Banach space
$$
\cA:=\{\phi\in C_0(\bR^3\times\bR^3)\text{ s.t. }(x,\xi)\mapsto\widehat{\phi(x,\cdot)}(\xi)\text{ belongs to }L^1(\bR^3_\xi;C_0(\bR^3_x))\}\,,
$$
the family $W_{\hbar}[R_{\hbar,N:1}(t)]$ is relatively compact in $\cA'$ for the weak-* topology by the Banach-Alaoglu theorem (Theorem 3.1.6 of \cite{Brezis}). By compactness and uniqueness of the limit point, 
we conclude that \eqref{Wignert} holds.

This completes the proof of Theorem \ref{T-MFSCLim}.


\section{Proof of Proposition \ref{P-SCLim}}\lb{S-SCMF}


This proof is similar to, and in places much simpler than the proof of Theorem \ref{T-MFSCLim}. Whenever similar to those in sections \ref{S-PropaEstim} and \ref{S-ProofT}, the arguments used 
in the proof of Proposition \ref{P-SCLim} are only sketched below. We shall insist only on those parts of the proof which which are simpler variants of sections \ref{S-PropaEstim} and \ref{S-ProofT}.

Consider the functional
$$
\ba
\cG[R_\hb,\rho,u](t):=\Tr_\fH(R_\hb(t)^{1/2}|\hb D_x-u(t,\cdot)|^2R_\hb(t)^{1/2})
\\
+\iint_{\bR^6}V(x-y)(\rho_\hb-\rho)(t,x)(\rho_\hb-\rho)(t,y)dxdy&\,,
\ea
$$
where $(\rho,u)$ is the solution of \eqref{EulP}, while $R_\hb$ is the classical solution of \eqref{Hartree}. 

We recall Hardy's inequality: for each $\phi\in H^1(\bR^3)$, one has
$$
\int_{\bR^3}\frac{|\phi(x)|^2}{|x|^2}dx\le 4\int_{\bR^3}|\grad_x\phi(x)|^2dx\,.
$$
We also recall Proposition 5.3 from \cite{Bove}:
$$
\ba
\Tr(R_\hb(t)^{1/2}(I-\Dlt_x)R_\hb(t)^{1/2})+\tfrac12\Tr(R_\hb(t)V\star\rho_\hb(t,\cdot))
\\
=\Tr(R_\hb(0)^{1/2}(I-\Dlt_x)R_\hb(0)^{1/2})+\tfrac12\Tr(R_\hb(0)V\star\rho_\hb(0,\cdot))
\ea
$$
for all $t\ge 0$ and $\hb>0$. Together with \eqref{TightR0} and Hardy's inequality, this implies that
$$
\sup_{t\ge 0}\Tr(R_\hb(t)^{1/2}(I-\Dlt_x)R_\hb(t)^{1/2})<\infty\,.
$$
Let $\psi_m(t,\cdot)$ with $m\ge 0$ be a complete orthonormal system of eigenfunctions of $R_\hb(t)$ in $\fH=L^2(\bR^3)$. Let $\l_m$ be the associated sequence of eigenvalues of $R_\hb(t)$, 
i.e. $R_\hb(t)\psi_m(t,\cdot))=\l_m\psi_m(t,\cdot)$ for each $m\ge 0$. Then
$$
\rho_\hb(t,x)=\sum_{m\ge 0}\l_m|\psi_m(t,x)|^2
$$
and, by the Sobolev embedding $H^1(\bR^3)\subset L^6(\bR^3)$, one has
$$
\ba
\|\rho_\hb(t,\cdot)\|_{L^3(\bR^3)}\le&\sum_{m\ge 0}\|\l_m|\psi_m(t,x)|^2\|_{L^3(\bR^3)}
\\
\le&C_S^2\sum_{m\ge 0}\|\sqrt{\l_m}\grad_x\psi_m(t,x)\|^2_{L^2(\bR^3)}
\\
\le&C_S^2\sup_{t\ge 0}\Tr(R_\hb(t)^{1/2}(I-\Dlt_x)R_\hb(t)^{1/2})\,.
\ea
$$
Since
$$
\rho_\hb(t,x)\ge 0\quad\text{ and }\quad\int_{\bR^3}\rho_\hb(t,x)dx=1\,,
$$
one has $\rho_\hb\in L^\infty(\bR_+;L^{6/5}(\bR^3))$ and therefore $t\mapsto V\star\rho_\hb(t,\cdot)\in L^\infty(\bR_+;L^6(\bR^3)$. In particular $t\mapsto\rho_\hb(t,\cdot)V\star\rho_\hb(t,\cdot)$
belongs to $L^\infty(\bR_+;L^1(\bR^3))$, and one has
$$
\iint_{\bR^6}V(x-y)(\rho_\hb-\rho)(t,x)(\rho_\hb-\rho)(t,y)dxdy\ge 0
$$
for all $t\in[0,T]$. (Indeed, since $\rho\in L^\infty([0,T];L^1\cap L^\infty(\bR^3))$, the difference in density functions  $\rho_\hb(t,\cdot)-\rho(t,\cdot)\in L^\infty([0,T];L^{6/5}(\bR^3))$, and approximating $\rho_\hb(t,\cdot)-\rho(t,\cdot)$
by a sequence of functions $\phi_n$ in the Schwartz class $\cS(\bR^3)$, we conclude by observing that
$$
\iint_{\bR^6}V(x-y)\phi_n(x)\phi_n(y)dxdy=\tfrac1{(2\pi)^3}\int_{\bR^3}\hat V(\xi)|\hat \phi_n(\xi)|^2d\xi\ge 0
$$
since $\hat V(\xi)=|\xi|^{-2}\ge 0$.) In particular $\cG[R_\hb,\rho,u](t)\ge 0$ for all $t\in[0,T]$. 

Proceeding as in the proof of Proposition \ref{P-DtE} shows that
$$
\ba
\cG[R_\hb,\rho,u](t)=\cG[R_\hb,\rho,u](t)
\\
-2\int_0^t\Tr_\fH(R_\hb(s)^{1/2}(\hb D_j-u_j(s,\cdot))\Si_{jk}(s,\cdot)(\hb D_k-u_k(s,\cdot))R_\hb(s)^{1/2})ds
\\
+\int_0^t\iint_{\bR^6}(u(s,x)-u(s,y))\cdot\grad V(x-y)(\rho_\hb-\rho)(s,x)(\rho_\hb-\rho)(s,y)dxdyds
\\
+\tfrac12\hbar^2\int_0^t\int_{\bR^3}\Dlt_x(\Div_xu(s,x))\rho_\hb(s,x)dxds&\,,
\ea
$$
where we recall the notation $\Sigma_{jk}=\tfrac12(\d_ju_k+\d_ku_j)$. As above
$$
\ba
\left|\int_{\bR^3}\Dlt_x(\Div_xu(s,x))\rho_\hb(s,x)dx\right|\le&\|\Dlt_x(\Div_xu)\|_{L^\infty([0,T\times\bR^3)}\int_{\bR^3}\rho_\hb(s,x)dx
\\
=&\|\Dlt_x(\Div_xu)\|_{L^\infty([0,T]\times\bR^3)}\,,
\ea
$$
and
$$
\ba
\left|\Tr_\fH(R_\hb(s)^{1/2}(\hb D_j-u_j(s,\cdot))\Si_{jk}(s,\cdot)(\hb D_k-u_k(s,\cdot))R_\hb(s)^{1/2})\right|
\\
\le\|\grad_xu\|_{L^\infty([0,T]\times\bR^3)}\Tr_\fH(R_\hb(s)^{1/2}|\hb D_x-u(s,\cdot)|^2R_\hb(s)^{1/2})&\,.
\ea
$$

Proceeding as in Lemma \ref{L-FcF}, one has
$$
\ba
\iint_{\bR^6}V(x_1-x_2)(\rho_\hb-\rho)(s,x_1)(\rho_\hb-\rho)(s,x_2)dx_1dx_2
\\
=\frac1N\iint_{\bR^6}V(x_1-x_2)\rho_\hb(s,x_1)\rho_\hb(s,x_2)dx_1dx_2
\\
+\int_{(\bR^3)^N}\frac1{N^2}F(X_N,\rho(s,\cdot))\prod_{n=1}^N\rho_\hb(s,x_n)dx_n&\,,
\ea
$$
and
$$
\ba
\iint_{\bR^6}(u(s,x_1)-u(s,x_2))\cdot\grad V(x_1-x_2)(\rho_\hb-\rho)(s,x_1)(\rho_\hb-\rho)(s,x_2)dx_1dx_2
\\
=\frac1N\iint_{\bR^6}(u(s,x_1)-u(s,x_2))\cdot\grad V(x_1-x_2)\rho_\hb(s,x_1)\rho_\hb(s,x_2)dx_1dx_2
\\
+\int_{(\bR^3)^N}\frac1{N^2}F'(X_N,(\rho,u)(s,\cdot))\prod_{n=1}^N\rho_\hb(s,x_n)dx_n&\,.
\ea
$$
It will be more convenient to recast these identities as
$$
\ba
\tfrac{N-1}N\iint_{\bR^6}V(x_1-x_2)(\rho_\hb-\rho)(s,x_1)(\rho_\hb-\rho)(s,x_2)dx_1dx_2
\\
=\int_{(\bR^3)^N}\frac1{N^2}F(X_N,\rho(s,\cdot))\prod_{n=1}^N\rho_\hb(s,x_n)dx_n
\\
+\tfrac1N\int_{\bR^3}(2\rho_\hb-\rho)(s,x)(V\star_x\rho)(s,x)dx&\,,
\ea
$$
and
$$
\ba
\tfrac{N-1}N\iint_{\bR^6}(u(s,x_1)-u(s,x_2))\cdot\grad V(x_1-x_2)(\rho_\hb-\rho)(s,x_1)(\rho_\hb-\rho)(s,x_2)dx_1dx_2
\\
=\tfrac1N\iint_{\bR^6}(u(s,x_1)-u(s,x_2))\cdot\grad V(x_1-x_2)(2\rho_\hb-\rho)(s,x_1)\rho(s,x_2)dx_1dx_2
\\
+\int_{(\bR^3)^N}\frac1{N^2}F'(X_N,(\rho,u)(s,\cdot))\prod_{n=1}^N\rho_\hb(s,x_n)dx_n&\,.
\ea
$$
Applying the Serfaty inequality shows that
$$
\ba
\Big|\tfrac{N-1}N\iint_{\bR^6}(u(s,x_1)\!-\!u(s,x_2))\cdot\grad V(x_1\!-\!x_2)(\rho_\hb\!-\!\rho)(s,x_1)(\rho_\hb\!-\!\rho)(s,x_2)dx_1dx_2
\\
-\tfrac1N\iint_{\bR^6}(u(s,x_1)-u(s,x_2))\cdot\grad V(x_1-x_2)(2\rho_\hb-\rho)(s,x_1)\rho(s,x_2)dx_1dx_2\Big|
\\
\le C\|\grad u(s,\cdot)\|_{L^\infty(\bR^3)}\Big(\tfrac{N-1}N\iint_{\bR^6}V(x_1-x_2)(\rho_\hb-\rho)(s,x_1)(\rho_\hb-\rho)(s,x_2)dx_1dx_2
\\
-\tfrac1N\int_{\bR^3}(2\rho_\hb-\rho)(s,x)(V\star_x\rho)(s,x)dx\Big)
\\
+C\left(1+\|\rho\|_{L^\infty(\bR^3)}\right)\left(1+\|u\|_{W^{1,\infty}(\bR^3)}\right)\frac1{N^{1/3}}&\,.
\ea
$$
Letting $N\to\infty$ in this inequality, we conclude that
$$
\ba
\Big|\iint_{\bR^6}(u(s,x_1)-u(s,x_2))\cdot\grad V(x_1-x_2)(\rho_\hb-\rho)(s,x_1)(\rho_\hb-\rho)(s,x_2)dx_1dx_2\Big|
\\
\le C\|\grad u(s,\cdot)\|_{L^\infty(\bR^3)}\iint_{\bR^6}V(x_1-x_2)(\rho_\hb-\rho)(s,x_1)(\rho_\hb-\rho)(s,x_2)dx_1dx_2&\,.
\ea
$$
Finally
$$
\ba
\frac{d}{dt}\cG[R_\hb,\rho,u](t)\le 2\|\grad_xu\|_{L^\infty([0,T]\times\bR^3)}\Tr_\fH(R_\hb(t)^{1/2}|\hb D_x-u(t,\cdot)|^2R_\hb(t)^{1/2})
\\
+C\|\grad u\|_{L^\infty([0,T]\times\bR^3)}\iint_{\bR^6}V(x_1-x_2)(\rho_\hb-\rho)(t,x_1)(\rho_\hb-\rho)(t,x_2)dx_1dx_2
\\
+\tfrac12\hb^2\|\Dlt_x(\Div_xu)\|_{L^\infty([0,T]\times\bR^3)}&\,,
\ea
$$
so that
$$
\ba
0\le\cG[R_\hb,\rho,u](t)\le\cG[R_\hb,\rho,u](0)e^{\max(2,C)t\|\grad_xu\|_{L^\infty([0,T]\times\bR^3)}}
\\
+\tfrac12\hb^2\|\Dlt_x(\Div_xu)\|_{L^\infty([0,T]\times\bR^3)}\frac{e^{\max(2,C)t\|\grad_xu\|_{L^\infty([0,T]\times\bR^3)}}-1}{\max(2,C)t\|\grad_xu\|_{L^\infty([0,T]\times\bR^3)}}&\,.
\ea
$$

With \eqref{Monokin0} and \eqref{EPrho0}, the previous inequality shows that
$$
\Tr_\fH(R_\hb(t)^{1/2}|\hb D_x-u(t,\cdot)|^2R_\hb(t)^{1/2})\to 0
$$
which is the third convergence statement to be proved in Proposition \ref{P-SCLim}, and
$$
\iint_{\bR^6}V(x-y)(\rho_\hb-\rho)(t,x)(\rho_\hb-\rho)(t,y)dxdy\to 0
$$
as $\hb\to 0$ for all $t\in[0,T]$. Since $\rho_\hb(t,\cdot)$ is a bounded family of Borel probability measures on $\bR^3$ for each $t\in[0,T]$ as $\hb$ runs through $[0,1]$, the same argument as in Step 2 of 
section \ref{S-ProofT} shows that $\rho_\hb(t,\cdot)\to\rho(t,\cdot)$ in the narrow topology of Radon measures on $\bR^3$. That $J_\hb(t,\cdot)\to\rho u(t,\cdot)$ for the narrow topology of signed Radon 
measures on $\bR^3$, and that $W_\hb[R_\hb(t)]\to\rho(t,x)\de(\xi-u(t,x))$ in $\cS'(\bR^3)$ for all $t\in[0,T]$ as $\hb\to 0$ is proved exactly as in steps 3 and 4 of section \ref{S-ProofT}. This concludes the 
proof of Proposition \ref{P-SCLim}.

\textbf{Remark.} This proof bears on a mean-field equation, describing the behavior of the typical single particle in a system of $N$ identical particles. It may seem somewhat strange that the argument
should involve $N$ as in the proof of \eqref{T-MFSCLim}, and one might have preferred a more direct argument. Nevertheless this one has the merit of being essentially a repetition of what has been
in the two previous sections.

\textbf{Acknowledgement.} We are grateful to S. Serfaty and M. Duerinckx for very useful additional explanations on their joint article \cite{DS}, of critical importance for our results.


\end{document}